 \documentclass[final,3p,times]{elsarticle}
\usepackage{bm}
\usepackage{caption}
\makeatletter

\newcommand{\Rmnum}[1]{\expandafter\@slowromancap\romannumeral #1@}
\makeatother
\usepackage{tabularx}
\usepackage{txfonts}
\usepackage{amsthm,amstext}
\usepackage{amsmath,amsfonts,amssymb,graphicx,mathrsfs,epsfig,subfigure,float}
\usepackage{algorithm}
\usepackage{algorithmic}
\usepackage{multirow}
\usepackage{times}
\usepackage{color}
\usepackage{epstopdf}  
\biboptions{sort&compress}

\newtheorem{theorem}{Theorem}
\newtheorem{corollary}{Corollary}
\newdefinition{definition}{Definition}
\newtheorem{lemma}{Lemma}
\newdefinition{proposition}{Proposition}
\newdefinition{assumption}{Assumption}
\newdefinition{example}{Example}
\newdefinition{case}{Case}
\newdefinition{remark}{Remark}

\begin{document}

\captionsetup[figure]{labelfont={bf},labelformat={default},labelsep=period,name={Fig.}}
\begin{frontmatter}
\title{Empathic network learning for multi-expert emergency decision-making under incomplete and inconsistent information}
\author[a]{Simin~Shen}
\ead{shensimin1013@163.com}
\author[a]{Zaiwu Gong\corref{cor1}}
\ead{zwgong26@163.com}
\author[a]{Bin Zhou}
\ead{binzhou@mail.ustc.edu.cn}
\author[b,c]{Roman S{\l}owi\'nski}
\ead{roman.slowinski@cs.put.poznan.pl}
\address[a]{Collaborative Innovation Center on Forecast and Evaluation of Meteorological Disasters, the Research Institute for Risk Governance and Emergency Decision-Making, School of Management Science and Engineering, Nanjing University of Information Science and Technology, Nanjing,210044,China}
\cortext[cor1]{Corresponding author}
\address[b]{Institute of Computing Science, Pozna\'n University of Technology, 60-965 Pozna\'n, Poland}
\address[c]{Systems Research Institute, Polish Academy of Sciences, 01-447 Warsaw, Poland}


\begin{abstract}
Challenges, such as a lack of information for emergency decision-making, time pressure, and limited knowledge of experts acting as decision-makers (DMs), can result in the generation of poor or inconsistent indirect information regarding DMs' preferences. Simultaneously, the empathic relationship represents a tangible social connection within the context of actual emergency decision-making, with the structure of the empathic network being a significant factor influencing the outcomes of the decision-making process. To deduce the empathic network underpinning the decision behaviors of DMs from incomplete and inconsistent preference information, we introduce an empathic network learning methodology rooted in the concept of robust ordinal regression via preference disaggregation. Firstly, we complete incomplete fuzzy judgment matrices including holistic preference information given in terms of decision examples on some reference alternatives, independently by each DM, and we calculate the intrinsic utilities of DMs. Secondly, we establish constraints for empathic network learning models based on empathic preference information and information about relations between some reference nodes. Then, the necessary and possible empathic relationships between any two DMs are calculated. Lastly, tailored to the specific requirements of different emergency scenarios, we design six target networks and construct models to derive the most representative empathic network.

\end{abstract}

\end{frontmatter}

\section{Introduction}
Unconventional emergencies, such as pandemics, earthquakes or fires, do not only pose a threat to human life and safety, but also bring a severe impact on the world economy. Due to the high destructiveness, complexity and variability associated with emergencies, it is necessary to coordinate the efforts of experts from various fields for effective collaboration in emergency decision-making. With the development of Internet technology, information dissemination and communication between individuals became more convenient. Therefore, when addressing real group decision-making problems, individuals may establish a social network with specific structural characteristics, and the structure of this social network is an important factor influencing the outcomes of decision-making \cite{1A.Godoy-Lorite2021,2J.Sun2023}. At present, the social relationships and their intensities among experts are usually reflected by the mutual relationship matrix constructed by experts using a subjective scoring method \cite{3S.M.Yu2021}. However, this approach may lead to excessive cognitive burden and information overload in time-critical emergency situations.

In actual emergency decision-making, a group of  experts need to assess a limited number of emergency alternatives and make decisions about them. As we know, emergency decision-making has a strong time limit, and experts tend to accept the opinions of other experts with whom they have social relations, which is conducive to improving the efficiency and quality of decision-making \cite{4X.Xu2020}. The decision-making behavior of experts, spanning from vaccine selection to voting preferences, depends on the structure of social networks \cite{1A.Godoy-Lorite2021}. However, experts are often only able to provide incomplete preference information for three reasons: first, due to the limitations of information, resources, and time, as well as expert cognition, it is difficult for experts to give a complete preorder of a subset of alternatives, but only a partial preorder, such as ``alternative $a$ is preferred to alternative $b$'' and other exemplary indirect preference information. Secondly, experts are often influenced by the opinions of other experts with whom they share empathic relationships, however, they have a partial understanding of the overall social network structure associated with the decision-making process. Therefore, only partial information about the network structure is known, such that, e.g.,  ``it is unknown if experts ${d_1}$ and ${d_2}$ have an empathic relationship''. Finally, because of the disagreement, the decision information provided by experts may be inconsistent. Therefore, the method proposed in this paper only requires experts to provide indirect preference information and partial node information, rather than asking them to provide a complete network structure, which can effectively reduce their cognitive load.

Empathy refers to an individual's capacity to understand the emotions of others and exhibit concern for the social welfare of others \cite{5P.Fontaine1997,6T.Singer2006}. This kind of phenomenon is common in the emergency decision making. Usually, in the process of emergency decision-making, the government  will give full consideration to the interests of the victims, and adopt specific measures for emergencies. For example, in order to minimize the economic damages from COVID-19 on workers and small businesses, the US federal government created the Paycheck Protection Program (PPP) \cite{7A.J.Staples2022}. In emergency decision-making, the empathic relationships between experts can help them consider issues from different perspectives, which is conducive to assessing alternatives and identifying potential risks, thereby making more comprehensive and effective decisions. Moreover, emergency decision-making often requires flexible responses to constantly changing situations. Empathy enables experts to better understand the behaviors and reactions of others, thereby enhancing adaptability and flexibility. Behavioral economics and neuroscience experiments have confirmed the existence of empathic preferences \cite{5P.Fontaine1997, 6T.Singer2006, 8J.Brandts2004, 9T.Singer2006}. In 2019, Salehi-Abari et al. \cite{10A.Salehi-Abari2019} introduced an empathic social choice framework for the first time in group decision-making, under which individual utility was decomposed into intrinsic utility and empathic utility to other individuals, and local and global empathic models were constructed. This work lays a theoretical foundation for the study of group decision making in the context of empathic network.

In the modern society, social networks are highly structured. The network structure directly affects its function and resilience \cite{11F.Kaiser2021}. In emergency circumstances, the network is extremely vulnerable to interference, and the breakage of a link may cause the paralysis of the communication in the network. Therefore, the research on network resilience in the field of emergency management has attracted extensive attention \cite{12X.Cao2022, 13S.R.Osman2023}. Resilience originated from systems ecology, and its basic meaning is the ability of a system to withstand external shocks and maintain its main functions in the event of a crisis \cite{14C.S.Holling1973}. Generally, the flatter the hierarchy of the network structure, the higher the degree of agglomeration and the more diverse the links between nodes, the stronger the anti-interference ability and the higher the resilience of the network \cite{15R.S.Burt1982}. In particular, if the influence of each node in the network is balanced, it is called a distributed network. The distributed network is decentralized and can avoid system paralysis caused by key nodes being attacked. Therefore, the distributed network system is widely used in the smart grid and the intelligent transportation system to improve system resilience \cite{16Y.Wang2020, 17X.Ge2020}. However, higher network resilience is not always better. For example, when in the face of disease transmission network, people tend to want it to be very low resilience. At the same time, higher network resilience means higher redundancy, which will reduce the efficiency of the network. When the influence of a node in the network is obvious compared with that of other nodes, it is called a central network. The central network is a typical network with low resilience but high efficiency. In addition, when the network structure is sparse enough, it can not only improve the network efficiency, but also meet the demand of rapid response.

Determining the network structure is usually the first step in any network analysis \cite{18L.Peel2022}. The research on social network inference first began in the field of sociology. In the 1970s, the problem of sampling and statistical inference of social networks was systematically raised, which refers to how to deduce group characteristics based on a small number of samples obtained from social networks \cite{19O.Frank1977,20O.Frank1980}. In the field of complex system, the application and research of social network sampling inference focus mainly on the description and estimation of topologies, attributes and other characteristics of large network systems, such as Facebook, Twitter, Sina Weibo and other online social networks \cite{21A.E.Mislove2009}. At present, there are many network-based statistical inference models, mainly including exponential random graph models \cite{22J.Xiong2020}, latent network models \cite{23R.van Bork2021} and stochastic block models \cite{24X.Lu2019}. It is worth mentioning that network statistical inference has become a hot field in systems biology \cite{25K.Faust2021,26N.F.Muller2022}. Biological network inference refers to the use of statistical methods to infer networks to predict new  biological mechanisms from partial prior knowledge such as interactions between biomolecules extracted from the analysis of large datasets. In addition, Benedetti et al. \cite{27E.Benedetti2020} proposed an optimization algorithm to infer the network having the maximum overlap with the available prior knowledge, and found that the obtained network was better than the statistical network.

In emergency decision-making scenarios, the incompleteness of information is a common issue. Due to time pressure and resource constraints, experts often find it difficult to provide comprehensive decision information. This lack of information poses a challenge in constructing an accurate decision-making model, as it may lead to inaccurate or impractical final decisions. Moreover, even when experts can provide some information, it is often fragmentary and incomplete. Against this backdrop, how to use limited information to infer network structures and thereby assist in forming effective decisions has become an urgent problem to be solved. This involves not only how to efficiently obtain key information from experts, but also how to infer network structures from this information and ensure that reliable and effective decision results can be generated based on these network structures. The motivation for our study arises from this challenge, with the aim of exploring a method that can infer network structures under incomplete information, by utilizing indirect preference information and partial node information provided by experts, to thereby assist in the emergency decision-making process. We proposes an empathic network learning method based on the idea of robust ordinal regression via preference disaggregation \cite{28S.Greco2008, 29J.R.Figueira2009, 30S.Greco2010}. Specifically:

\begin{itemize}
  \item According to the incomplete decision information of the experts and the nature of the empathic network itself, the linear constraints are formulated defining the set of compatible empathic networks. The incomplete decision information includes indirect preference information and partial node information. This decision information requires only low cognitive effort. In addition, for the case that the decision information provided by experts may be inconsistent, a mixed $\{0,1\}$ linear programming model inspired by \cite{30S.Greco2010} is formulated to find out and eliminate the minimum number of inconsistent constraints.

  \item Based on the idea of necessary and possible preference relations proposed in the robust ordinal regression approach  \cite{28S.Greco2008}, we give in this paper definitions of the necessary empathic and the possible empathic relationships in the network. The necessary empathic relationship holds for any two experts if and only if all compatible empathic networks have admitted they have an empathic relationship. A possible empathic relationship holds for any two experts if and only if there is at least one compatible empathic network that admits that they have an empathic relationship. The empathic network learning models are constructed to evaluate the possible and necessary empathic relationships between any two experts.

  \item We have defined the set of compatible empathic networks representing all possible empathic network structures and have studied how to select the most representative network structure from the compatible set. we provide experts with six types of target networks, i.e., the most discriminating empathic network, sparse empathic network, central and distributed empathic network, resilient empathic network and empathic networks with star, bus and tree topologies, and we set corresponding models to select a particular target network from the set of compatible empathic networks inspired by \cite{31M.Kadzinski2013}. The whole procedure ensures the robustness of the end results.
\end{itemize}

The remainder of the paper is organized as follows. Section 2 introduces the basic properties of empathic networks and the types of incomplete decision information. The empathic network learning method is put forward in Section 3. Section 4 presents a numerical example and Section 5 provides conclusions.

\section{Definitions and notation}
\subsection{Empathic networks}
Taking into account that decision information provided by DMs in different emergency decision-making situations is incomplete, this paper infers empathic relationships between DMs. Let us consider a finite set of alternatives $A = \{ {a_1},{a_2}, \cdots ,{a_m}\} $ and its index set $M=\{ 1,2, \cdots ,m\}$, as well as a finite set of DMs $D = \{ {d_1},{d_2}, \cdots ,{d_n}\}$ among which there exist empathic relationships; ${d_i}$ has empathy for ${d_j}$ indicates that ${d_i}$ is influenced by preferences of ${d_j}$. In this paper, the directed weighted graph $G(N,E)$ is used to represent the empathic network composed of all DMs present in the nodes $N = \{1,2, \cdots,n\}$. The directed arc $e(i, j)$ in $E$ indicates that ${d_i}$ has an empathic relation to ${d_j}$. The weight information on the directed arc is the empathic weight, which represents the intensity of the empathic relationship between ${d_i}$ and ${d_j}$. Clearly, an empathic network corresponds to an $n \times n$ empathic matrix. That is:

\begin{equation*}
W=\begin{pmatrix}
{{w_{11}}}&{{w_{12}}}& \cdots &{{w_{1n}}}\\
{w{}_{21}}&{{w_{22}}}& \cdots &{{w_{2n}}}\\
 \vdots & \vdots &{}& \vdots \\
{{w_{n1}}}&{w{}_{n2}}& \cdots &{{w_{nn}}}\\
\end{pmatrix},
\end{equation*}
where ${w_{ij}} \ge 0$ $(i \ne j)$ indicates that the empathic relationships between DMs have positive effects, and there is no jealousy or malice; ${w_{ii}} > 0$, $i= 1,2, \cdots ,n$ indicates that the DMs are not completely altruistic, and they retain part of their intrinsic preferences while paying attention to the preferences of their immediate neighbors. Moreover, row sum of the empathic matrix is 1, that is, ${w_{jj}} + \sum\limits_{k \ne j} {{w_{jk}}}  = 1$, indicating that the sum of the influences of ${d_j}$'s immediate neighbors on ${d_j}$ is 1.

\begin{definition}\cite{10A.Salehi-Abari2019}
 The $j$-th column sum of the empathic matrix is equal to the empathic centrality ${\omega _j}$ of ${d_j}$. That is
\begin{equation}
{\omega _j} = \sum\limits_{k = 1}^n {{w_{kj}}},
\end{equation}
indicating the influence of ${d_j}$ in the entire empathic network.
\end{definition}

Greater empathic centrality ${\omega _j}$ means greater influence of ${d_j}$ in the empathic network. It is easy to observe that ${\omega _1} + {\omega _2} +  \cdots  + {\omega _n} = n$.

Considering the transitivity of empathic relationships, Salehi-Abari et al. \cite{10A.Salehi-Abari2019} introduced the concept of local empathic and global empathic networks. The local empathy means that an individual pays attention to the preferences of only neighboring individuals, while global empathy means that an individual pays attention to the preferences of all individuals. It is common to define the above empathic matrix $W$ as a local empathic weight matrix. In a local empathic network, the empathic utility ${u_j}(a_s)$ of ${d_j}$ about alternative ${a_s} \in A, s \in M$, after being influenced by the empathic relationships is
\begin{equation}
{u_j}(a_s) = {w_{jj}}u_j^I(a_s) + \sum\limits_{k \in {N_j}} {{w_{jk}}} u_k^I(a_s).
\end{equation}

The corresponding matrix form is
\begin{equation}
U\left( a_s \right) = W{U^I}(a_s),
\end{equation}
where $U(a_s) = {({u_1}(a_s),{u_2}(a_s), \cdots {u_n}(a_s))^T}$,  $ \ {U^I}(a_s) = {(u_1^I(a_s),u_2^I(a_s), \cdots u_n^I(a_s))^T}$.

In global empathic networks, the empathic utility of ${d_j}$ about alternative $a_s \in A, s \in M$ after being influenced by the empathic relationships is
\begin{equation}
{u_j}\left( a_s \right) = {w_{jj}}u_j^I(a_s) + \sum\limits_{k \in {N_j}} {{w_{jk}}} u_k(a_s).
\end{equation}

The corresponding matrix form is
\begin{equation}
U\left( a_s \right) = {(I - W + D)^{ - 1}}D{U^I}(a_s),
\end{equation}
where $I$ is an $n \times n$ identity matrix and $D$ is an $n \times n$ diagonal matrix with ${d_{jj}} = {w_{jj}},j \in N$. In particular, $I - W + D$ is invertible as guaranteed by ${w_{ij}} \ge 0$, $ \ i,j\in N$, $ \  i \ne j$, $\ {w_{jj}} > 0$, $ \ j\in N$ and $ \ {w_{jj}} + \sum\limits_{k \ne j} {{w_{jk}}}  = 1$ \cite{10A.Salehi-Abari2019}. Then
\begin{equation}
U\left( a_s \right) = G{U^I}(a_s),
\end{equation}
where
\begin{equation}
G = {(I - W + D)^{ - 1}}D
\end{equation}
is the global empathic weight matrix.

\subsection{Fuzzy judgment matrix and eigenvector method}\label{eigenvector}
DMs usually express preferences in the form of pairwise comparisons and construct a fuzzy judgment matrix. Each DM provides the comparative preference information for set $A$ of alternatives \cite{32T. Tanino1984}, where the preference judgment is represented by a real value in the fuzzy judgment matrix.

\begin{definition}\cite{32T. Tanino1984}
If the pairwise comparison matrix $R = {({r_{st}})_{m \times m}}$ has the following properties:

(a) ${r_{ss}} = 0.5$, $\forall s \in M$;

(b) ${r_{st}} + {r_{ts}} = 1$, $\forall s,t \in M$,

\noindent then $R$ is referred to as a fuzzy judgment matrix, where the element $r_{st}$ represents the degree to which alternative ${a_s}$ is preferred to alternative ${a_t}$. ${r_{st}} = 0.5$ indicates that there is indifference between ${a_s}$ and ${a_t}$; $0 \le {r_{st}} < 0.5$ means that ${a_t}$ is preferred to ${a_s}$; $0.5 < {r_{st}} \le 1$ means that ${a_s}$  is preferred to ${a_t}$.
\end{definition}

\begin{definition}\cite{32T. Tanino1984}
Let $R = {({r_{st}})_{m \times m}}$ be a fuzzy judgment matrix. If matrix $R$ satisfies the following property:
\begin{equation}
{r_{ik}} + {r_{kj}} = {r_{ij}} + 0.5, \ i,k,j \in M,
\end{equation}
then $R$ is said to satisfy additive consistency.
\end{definition}

Saaty \cite{33T. L. Saaty1977} applied the eigenvector method in the Analytic Hierarchy Process (AHP) to calculate the maximum eigenvalue and the corresponding eigenvector of the judgment matrix. In this paper, we use the eigenvector method to calculate the maximum eigenvalue and the corresponding eigenvector of the fuzzy judgment matrix $R$ according to the following equation:
\begin{equation}
R{\bf{w}} = {\lambda _{\max }}{\bf{w}},
\end{equation}
where $\lambda _{\max }$ is the maximum eigenvalue of $R$, and $\bf{w}$ is the corresponding eigenvector.

\subsection{Decision information}
Decision information provided by DMs is composed of preference information concerning reference alternatives and network node information about empatic relations between DMs. In the proposed method, since the calculation of intrinsic utilities is necessary, all DMs provide the fuzzy judgment matrices independently on the whole set $A$, without being affected by the empathic relationships with each other. If the DMs can provide the complete fuzzy judgment matrices, then we use the eigenvector method to obtain the intrinsic utilities of the DMs. Note, however, that due to urgency of emergency decision-making and the limited cognition of DMs, the fuzzy judgment matrices provided by the DMs may be incomplete. In this case, we can use the robust ordinal regression approach to complete the incomplete judgment matrices and then calculate intrinsic utilities. Therefore, the DMs can provide incomplete decision information, that is, indirect preference information on subsets of alternatives. This information is then used to complete the incomplete judgment matrices.

Once intrinsic utilities are obtained, the DMs are allowed to communicate and the empathic relationships between them affect their preferences. Thus, the DMs are asked to provide incomplete decision information composed of indirect preference information on subsets of alternatives and partial network node information on subsets of network nodes.

First, we describe how we capture intrinsic utilities when DMs can provide incomplete judgment matrices only. Each DM makes independently pairwise comparisons of the alternatives from set $A$, giving an $m \times m$ judgment matrix of alternatives, representing their own intrinsic preferences. Due to the complexity of emergency decision-making and the limitation of expert cognition, some upper triangular elements of the judgment matrix may be missing. Suppose that each DM provides an incomplete fuzzy judgment matrix  $R = {({r_{st}})_{m \times m}}$ on set $A$, satisfying ${r_{ss}} = 0.5$, ${r_{st}} \ge 0$, ${r_{st}} + {r_{ts}} = 1$, $s,t \in M $, having the following form \cite{34Z.Xu2006}:
\begin{equation}
R = \left( {\begin{array}{*{20}{c}}
0.5&*&{{r_{13}}}& \cdots &{{r_{1m}}}\\
*&0.5&*& \cdots &{{r_{2m}}}\\
{{r_{31}}}&*&0.5& \cdots &*\\
*&{{r_{m - 1,2}}}&*& \ddots &{{r_{m - 1,m}}}\\
{{r_{m1}}}&{{r_{m2}}}&*& \cdots &0.5
\end{array}} \right),
\end{equation}
where $*$ is the missing term of $R$. Furthermore, we assume that the incomplete fuzzy judgment matrix satisfies the additive consistency ${r_{ik}} + {r_{kj}} = {r_{ij}} + 0.5, \ i,k,j \in M$ \cite{35S.Alonso2004}.

Usually, due to the heterogeneity of knowledge and diversity of opinions, it is difficult for DMs to provide reliable direct preference information about parameters of their preference model, but they are usually able to provide exemplary indirect preference information, such as pairwise comparisons between alternatives \cite{28S.Greco2008}. Therefore, we assume that each DM can define a subset of reference alternatives $A_j^R = \{{a_{1_j}},{a_{2_j}}, \cdots ,{a_{m_j}}\}  \subseteq A$ and express their  preferences and preference intensity information through the partial preorder $\succeq$ in the set of reference alternatives and the partial preorder $\succeq^{*}$ underlying the intensity of preference among pairs of reference alternatives, as proposed in \cite{29J.R.Figueira2009}, as shown below:

\begin{itemize}
\item a partial preorder $\succeq^{j}$ on ${A_j^R}$, $j \in N$, whose semantics for $({a_s},{a_t})\in  {A_j^R}$ is, ${a_s}\succeq^{j}{a_t}\Leftrightarrow$ ${d_j}$ believes ${a_s}$ is at least as good as ${a_t}$.
\end{itemize}

\begin{itemize}
\item a partial preorder $\succeq^{*j}$ on ${A_j^R}$, $j \in N$, whose semantics for $ ({a_s},{a_t}),({a_p},{a_q})\in {A_j^R}$ is, $({a_s}, {a_t}) \succeq^{*j}({a_p},{a_q})\Leftrightarrow$ ${d_j}$ believes that ${a_s}$ is preferred to ${a_t}$ at least as much as ${a_p}$ is preferred to ${a_q}$.
\end{itemize}

Based on this information, the robust ordinal regression method \cite{28S.Greco2008, 29J.R.Figueira2009, 30S.Greco2010} is used to complete the incomplete fuzzy judgment matrices, and then the intrinsic utilities of each DM are obtained using the eigenvector method.

Then, intrinsic utilities of all DMs are made public, and communication between DMs is allowed. At this stage, the preferences of the DMs will change under the influence of the empathic relationships. We assume that the DMs take ${A^R} = \mathop  \bigcup \limits_{j \in N} A_j^R \subseteq A$ as a subset of reference alternatives, and express on $A^R$ their empathic preferences, and on $A^R \times A^R$ the information on the intensity of empathic preferences, through the partial preorders $\succeq$ and $\succeq^{*}$, respectively. It should be noted that if DMs provide the complete fuzzy judgment matrices, then they do not need to provide $A_j^R$, but only $A^R$ directly. In addition, the DMs can also provide a set of reference nodes ${N^R} \subseteq N$ to give node information in the empathic network, such as ``node $i$ has about the same degree of empathy for nodes $j$ and $k$" or ``the influence of node $l$ on node $i$ is more than a half of the empathic centrality of node $l$".

Further, consider the preference relation $\succeq$ and the intensity of preference relation $\succeq^{*}$ on ${A^R}$, both of which are weak preference relations that can be decomposed into symmetric part $\sim$ ($\sim^{*}$) and asymmetric part $\succ$ ($\succ^{*}$). For example, $\forall {a_s},{a_t},{a_p},{a_q}\in {A^R}$,
\begin{itemize}
  \item ${a_s} \succ {a_t} \equiv [{a_s}\succeq{a_t} \hspace{0.1cm} \mbox{and} \hspace{0.1cm} not({a_t}\succeq{a_s})] \Leftrightarrow$ ${a_s}$ is preferred to ${a_t}$.
  \item ${a_s} \sim {a_t} \equiv [{a_s}\succeq{a_t} \hspace{0.1cm} \mbox{and} \hspace{0.1cm} {a_t}\succeq{a_s}] \Leftrightarrow$ ${a_s}$ is indifferent to ${a_t}$.
  \item $({a_s},{a_t}){ \succ ^*}({a_p},{a_q}) \equiv [({a_s},{a_t})\succeq^*({a_p},{a_q}) \hspace{0.1cm} \mbox{and} \hspace{0.1cm} not(({a_p},{a_q})\succeq^*({a_s},{a_t}))] \Leftrightarrow$ the degree to which ${a_s}$ is preferred to ${a_t}$ is stronger than the degree to which ${a_p}$ is preferred to ${a_q}$.
   \item $({a_s},{a_t}){ \sim ^*}({a_p},{a_q}) \equiv [({a_s},{a_t})\succeq^*({a_p},{a_q}) \hspace{0.1cm} \mbox{and} \hspace{0.1cm}
       ({a_p},{a_q})\succeq^*({a_s},{a_t})] \Leftrightarrow$  there is no difference in the degree to which ${a_s}$ is preferred to ${a_t}$ and the degree to which ${a_p}$ is preferred to ${a_q}$.
\end{itemize}

\section{Introduction of the empathic network learning method}
In this section, we present a comprehensive description of the empathic network learning methodology, including the elicitation of incomplete decision information by DMs, calculation of intrinsic utilities, three types of constraints on network learning models, the computation of binary empathic relations and the identification and removal of inconsistent information, as well as how to select representative empathic networks from the set of compatible empathic networks. The flow chart of the empathic network learning method is shown in Fig.~\ref{f}.
\begin{figure}[!h]
\centering
\includegraphics[width=0.6\textwidth]{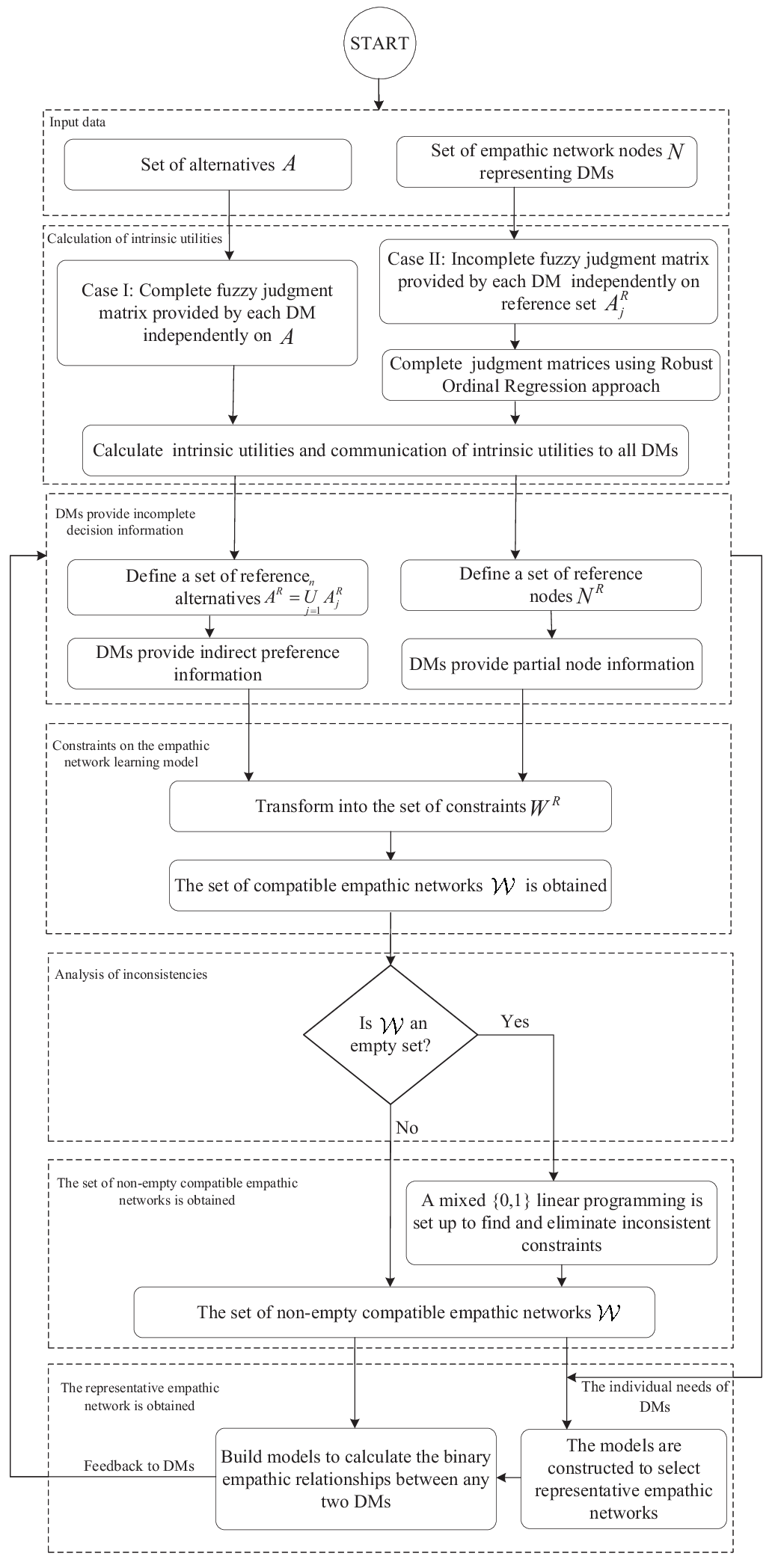}
\caption{The flow chart of the empathic network learning method.}\label{f}
\end{figure}

\subsection{Elicitation of incomplete decision information}
In our method, DMs provide incomplete holistic decision information on ${A^R}$ and ${N^R}$.

\vspace{6pt}

(a) Preference information on ${A^R}$.

\begin{itemize}
\item a partial preorder $\succeq^{d_j}$ on ${A^R}$, $ {d_j} \in D$, whose semantics for $({a_s},{a_t})\in  {A^R}$ is, ${a_s}\succeq^{d_j}{a_t}\Leftrightarrow$ ${d_j}$ believes ${a_s}$ is at least as good as ${a_t}$.
\end{itemize}

\begin{itemize}
\item a partial preorder $\succeq^{*d_j}$ on ${A^R}$, $ {d_j} \in D$, whose semantics for $ ({a_s},{a_t}),({a_p},{a_q})\in {A^R}$ is, $({a_s}, {a_t}) \succeq^{*d_j}({a_p},{a_q})\Leftrightarrow$ ${d_j}$ believes that ${a_s}$ is preferred to ${a_t}$ at least as much as ${a_p}$ is preferred to ${a_q}$.
\end{itemize}

(b) Partial node information on ${N^R}$.

\begin{itemize}
\item  ${w_{ij}} = 0, \ i,j \in N^R$, means that ${d_i}$ is not empathic to ${d_j}$, $\ {d_i},{d_j} \in D$.
\end{itemize}

\begin{itemize}
\item ${w_{ij}} \ge {w_{kh}}, \ i,j,k,h \in N^R$, means that the empathic intensity of ${d_i}$ to ${d_j}$ is at least as strong as the empathic intensity of ${d_k}$ to ${d_h}$, $\ {d_i},{d_j},{d_k},{d_h} \in D$.
\end{itemize}

\begin{itemize}
\item $\frac{{{w_{ij}}}}{{{\omega _j}}} \ge \frac{1}{2}, \ i,j \in N^R$, means that the empathic influence of ${d_j}$ to ${d_i}$ is not less than half of the total influence of ${d_j}$, $\ {d_i},{d_j} \in D$.
\end{itemize}

\begin{itemize}
\item ${\omega _i} - {\omega _j} \ge {\omega _k} - {\omega _h}, \ i,j,k,h \in N^R$, means that the empathic influence of ${d_i}$ exceeds that of ${d_j}$ by at least as much as the influence of ${d_k}$ exceeds that of ${d_h}$, $\ {d_i},{d_j},{d_k},{d_h} \in D$.
\end{itemize}

\subsection{Calculation of intrinsic utilities}
For the DMs' intrinsic utilities of alternatives, we consider two cases: (1) DMs provide complete fuzzy judgment matrices independently. (2) DMs provide incomplete fuzzy judgment matrices independently. We present a way to obtain intrinsic utilities based on the preference disaggregation method and fuzzy judgment matrix theory. In the following, we will detail the calculation of intrinsic utilities in two cases.

Case \Rmnum {1}: DMs provide complete fuzzy judgment matrices independently.

We use the eigenvector method of subsection~\ref{eigenvector} to obtain the eigenvector corresponding to the maximum eigenvalue of the fuzzy judgment matrix, and take this eigenvector as the intrinsic utilities of the DM about all alternatives. The intrinsic utility matrix  $U^I$ has the following form:
\begin{equation}
{U^I} = \left( {\begin{array}{*{20}{c}}
{U_1^I({{a_1}})}&{U_1^I({{a_2}})}& \cdots &{U_1^I({{a_m}})}\\
{U_2^I({{a_1}})}&{U_2^I({{a_2}})}& \cdots &{U_2^I({{a_m}})}\\
 \vdots & \vdots &{}& \vdots \\
{U_n^I({{a_1}})}&{U_n^I({{a_2}})}& \cdots &{U_n^I({{a_m}})}
\end{array}} \right),
\end{equation}
where $U_j^I({{a_s}})$ is an intrinsic utility of ${d_j}$ with respect to alternative ${a_s}$, $\forall j \in N$, $\forall {d_j} \in D$, $\forall {a_s} \in A$.

Case \Rmnum {2}: DMs provide incomplete fuzzy judgment matrices independently.

Based on holistic preference information provided by each DM $d_j$ on $A_j^R$ in the form of the incomplete fuzzy judgment matrix satisfying additive consistency, we apply the robust ordinal regression approach to complete the incomplete fuzzy judgment matrix through resolution of the following linear programming problem:
\begin{align*}
\begin{array}{r@{~}l}
\quad &\text{(Model 1)}\hspace{0.1cm}\max \varepsilon \\[10pt]
\quad &\text{\textup{subject to}}
\begin{cases}
\hspace{0.1cm} {r_{st}} \ge 0.5 +\varepsilon , \mbox{ if }  {a_s}{\succ ^j}{a_t} \hspace{0.1cm} \text { for } {a_s},{a_t}\in {A_j^R},j \in N,\\[10pt]
\hspace{0.1cm} {r_{st}} \ge {r_{pq}} + \varepsilon , \mbox{ if }  ({a_s}, {a_t}) \succeq^{*j}({a_p},{a_q}) \hspace{0.1cm} \text { for } {a_s},{a_t}\in {A_j^R},j \in N,\\[10pt]
\hspace{0.1cm} {r_{sp}} + {r_{pt}} = {r_{st}} + 0.5  \hspace{0.1cm} \text {for}\hspace{0.1cm} s < t < p, s,t,p\in M,\\[10pt]
\hspace{0.1cm} 0 \le {r_{st}} \le 1  \hspace{0.1cm} \text {for} \hspace{0.1cm} s,t \in M,
\end{cases}
\end{array}
\end{align*}
where $\varepsilon > 0$ is a slack variable. If the optimal value ${\varepsilon ^*} > 0$ is obtained, this indicates that the incomplete fuzzy judgment matrix can be completed. If Model 1 does not have a feasible solution, or ${\varepsilon ^*} \le 0$, this indicates that the indirect preference information provided by the DM $d_j$ is inconsistent. In this case, we introduce $\{0,1\}$ variables ${\nu_r}, \ r = 1, \cdots ,l$, and build a model to find and eliminate inconsistent preference information. This specific operation is similar to that in subsection~\ref{inconsistencies}. Due to space limitations, we do not present it here. After obtaining the complete fuzzy judgment matrix of each DM, we still use the eigenvector method to calculate the intrinsic utility matrix.

\subsection{Constraints and binary empathic relations}
In this subsection, we formulate the constraints of the empathic network learning models that characterize the set of compatible empathic networks. On this basis, we calculate the binary relations between DMs.

\subsubsection{Constraints on the empathic network learning model}
Utilizing the preference information and node information provided by the DMs, we can formulate the set of constraints ${W^R}$ that the empathic network should satisfy. In the following, we present three categories of constraints.

Firstly, we have to ensure that empathic preferences and empathic utility of the DMs are consistent. Then, according to the empathic preference information and the intensity of preference provided by every ${d_j} \in D$, we get the set of constraints ${W^{{A^R}}}$ as follows:

(a) ${u_j}({{a_s}}) \ge {u_j}({{a_t}}) + \varepsilon $, if ${a_s}\succ^{d_j}{a_t}$,

(b) ${u_j}({{a_s}}) = {u_j}({{a_t}})$, if ${a_s}\sim^{d_j}{a_t}$,

(c) ${u_j}({{a_s}}) - {u_j}({{a_t}}) \ge {u_j}({{a_p}}) - {u_j}({{a_q}}) + \varepsilon $, if $({{a_s},{a_t}}){ \succ ^{*{d_j}}}({{a_p},{a_q}})$,

(d) ${u_j}({{a_s}}) - {u_j}({{a_{t}}}) = {u_j}({{a_{p}}}) - {u_j}({{a_{q}}}) $, if $({{a_s}, {a_{t}}}){ \sim ^{*{d_j}}}( {{a_{p}},{a_{q}}})$,

\noindent where ${a_s},{a_t},{a_p},{a_q}\in {A^R}$, and $\varepsilon > 0$ is a slack variable.

Secondly, according to the node information provided by the DMs, the set of constraints ${W^{{N^R}}}$ can be written as follows:

(e) ${w_{ij}} = 0$,

(f) ${w_{ij}} \ge {w_{kh}} + \varepsilon $, if ${w_{ij}} > {w_{kh}}$,

(g) $\frac{{{w_{ij}}}}{{{\omega _j}}} \ge \frac{1}{2}$,

(h) ${\omega _i} - {\omega _j} \ge {\omega _k} - {\omega _h} + \varepsilon $, if ${\omega _i} - {\omega _j} > {\omega _k} - {\omega _h}$,

\noindent where $i,j,k,h \in {N^R}$, and $\varepsilon $ is the same as above.

Finally, according to the properties of the empathic matrix itself, we can write the set of basic constraints ${W^{base}}$ as follows:

(i) $\sum\limits_{k = 1}^n {{w_{jk}}}  = 1$, $\ \forall j \in N$,

(j) $w_{ij}^{} \ge 0$, $\ \forall i,j \in N \hspace{0.1cm} and \hspace{0.1cm} i \ne j$,

(k) ${w_{jj}} \ge \varepsilon^\prime$, $ \ \forall j \in N$,

\noindent where ${w_{jj}} > 0$ is written as ${w_{jj}} \ge \varepsilon^\prime$. The reason is that when the intensity of the empathic relationship between two nodes is close to 0, it is difficult to reflect the empathic relationship and play a role in the evolution of the opinions in the network. Therefore, we set a non-zero empathic relationship intensity threshold $\varepsilon^\prime > 0$ to better capture the empathic relationships between nodes.

To sum up, the following set of constraints ${W^R}$ defines an empathic network able to reproduce the preference and node information provided by DMs:

\begin{align*}
W^R
\begin{array}{r@{~}l}
\begin{cases}
W^{{A^R}}
&\begin{cases}
\hspace{0.1cm}\quad &{u_j}({{a_s}}) \ge {u_j}({{a_t}}) + \varepsilon , \mbox{ if }  {a_s}{ \succ ^{{d_j}}}{a_t} \hspace{0.1cm} \text { for } {a_s},{a_t}\in {A^R},{d_j} \in D, \\
\hspace{0.1cm}\quad &{u_j}({{a_s}}) = {u_j}({{a_t}}), \mbox{ if } {a_s}{ \sim ^{{d_j}}}{a_t}
\hspace{0.1cm} \text { for } {a_s},{a_t}\in {A^R},{d_j} \in D, \\
\hspace{0.1cm} \quad & {u_j}({{a_s}}) - {u_j}({{a_t}}) \ge {u_j}({{a_p}} ) - {u_j}( {{a_q}})+ \varepsilon, \mbox{ if } ({{a_s},{a_t}}){ \succ ^{*{d_j}}}({{a_p},{a_q}})\\
&\hspace{0.1cm} \text { for } {a_s},{a_t},{a_p},{a_q}\in {A^R},{d_j} \in D,\\
\hspace{0.1cm} \quad & {u_j}({{a_s}}) - {u_j}({{a_t}}) = {u_j}({{a_p}}) - {u_j}( {{a_q}}) , \mbox{ if } ({{a_s},{a_t}}){ \sim^{*{d_j}}}({{a_p},{a_q}})\\
&\hspace{0.1cm} \text { for } {a_s},{a_t},{a_p},{a_q}\in {A^R},{d_j} \in D,\\
\hspace{0.1cm}\quad & \cdots
\end{cases}\\[10pt]
W^{{N^R}}
&\begin{cases}
\hspace{0.1cm} \quad &{w_{ij}} = 0 \hspace{0.1cm} \text { for } i,j\in {N^R},\\
\hspace{0.1cm}\quad &{w_{ij}} \ge {w_{kh}} + \varepsilon  \hspace{0.1cm} \text { for } i,j,k,h\in {N^R},\\
\hspace{0.1cm}\quad &\frac{{{w_{ij}}}}{{{\omega _j}}} \ge \frac{1}{2} \hspace{0.1cm} \text { for } i,j\in {N^R},\\
\hspace{0.1cm} \quad & {\omega _i} - {\omega _j} \ge {\omega _k} - {\omega _h} + \varepsilon \hspace{0.1cm} \text { for } i,j,k,h\in {N^R},\\[15pt]
\hspace{0.1cm}\quad & \cdots
\end{cases}\\[10pt]
W^{base}
&\begin{cases}
\hspace{0.1cm} \quad &\sum\limits_{k = 1}^n {{w_{jk}}}  = 1, \ \forall j \in N,\\
\hspace{0.1cm}\quad &w_{ij}^{} \ge 0, \ \forall i,j \in N \mbox{ and } i \ne j,\\
\hspace{0.1cm} \quad & {w_{jj}} \ge \varepsilon^\prime, \ \forall j \in N.
\end{cases}
\end{cases}
\end{array}
\end{align*}

Therefore, an empathic network satisfying ${W^R}$ is called a compatible empathic network, and the set of compatible empathic networks is denoted by $\mathcal{W}$.

\subsubsection{Necessary and possible empathic relationships}
This subsection provides definitions of binary empathic relationships including necessary and possible empathic relationships, as well as the models to calculate the binary empathic relationships between any two DMs.

\begin{definition}
For any two DMs $\ {d_i},{d_j} \in D$, if all the empathic networks in the set of compatible empathic networks $\mathcal{W}$ satisfy ${w_{ij}} > 0$, $i,j \in N$, then ${d_i}$ has a necessary empathic relationship with ${d_j}$, denoted by ${d_i}{N^l}{d_j}$.
\end{definition}
\begin{definition}
For any two DMs $\ {d_i},{d_j} \in D$, if there exists at least one empathic network satisfying ${w_{ij}} > 0$, $i,j \in N$ in the set of compatible empathic networks $\mathcal{W}$, then ${d_i}$ has a possible empathic relationship with ${d_j}$, denoted by ${d_i}{P^l}{d_j}$
\end{definition}

In order to calculate the binary empathic relationships ${N^l}$ and ${P^l}$, for $\forall {d_i},{d_j} \in D$, we propose the following two models:
\begin{align*}
\begin{array}{r@{~}l}
\quad &\text{(Model 2)}\hspace{0.1cm}\max \varepsilon \\[18pt]
\quad &\text{\textup{subject to}}
\begin{cases}
\hspace{0.1cm} {w_{ij}} = 0 \hspace{0.1cm} \text { for } i,j \in N,\\[18pt]
\hspace{0.1cm} {W^R}.
\end{cases}
\end{array}
\end{align*}

If Model 2 has no solution or optimal solution ${\varepsilon ^*} < 0$, then ${d_i}{N^l}{d_j}$.

\begin{align*}
\begin{array}{r@{~}l}
\quad &\text{(Model 3)}\hspace{0.1cm}\max \varepsilon \\[18pt]
\quad &\text{\textup{subject to}}
\begin{cases}
\hspace{0.1cm} {w_{ij}} \ge \varepsilon^\prime \hspace{0.1cm} \text { for } i,j \in N,\\[18pt]
\hspace{0.1cm} {W^R},
\end{cases}
\end{array}
\end{align*}
where ${w_{ij}} \ge \varepsilon^\prime$ is replacing ${w_{ij}} > 0$, and $\varepsilon^\prime$ is the empathic relationship intensity threshold. If ${\varepsilon ^*} > 0$ is the optimal solution of Model 3, then ${d_i}{P^l}{d_j}$.

Calculated by Model 2 and Model 3, the binary empathic relationships $N^l$ and $P^l$ can be seen as results of the current phase of interaction with the DMs. With this information, DMs can deepen their understanding of the overall network structure and, in consequence,  they may wish to provide additional preference and node information. This interactive decision-making process promotes accuracy and reliability in decision support.

\subsubsection{Analysis of inconsistencies}\label{inconsistencies}
Observe that the set of feasible solutions to ${W^R}$ corresponding to the set of compatible empathic networks $\mathcal{W}$ may be empty, and one of the most likely reasons is that there are inconsistencies in the preference information and the node information given by the DMs. In this case, the decision analyst can interactively help the DMs to remove the piece of troublesome information about preferences concerning the reference alternatives, or information about the intensities of the empathic relationships, or information about the nodes. In order to identify these inconsistencies, we formulate a mixed $\{0,1\}$ linear programming problem. To this end, we introduce $\{0,1\}$ variables ${\nu_r}, \ r = 1, \cdots ,l$, corresponding to constraints of ${W^R}$ other than ${W ^ {base}}$. When variable  ${\nu_r}, \ r \in \{1, \cdots ,l\}$ is equal to 1, then the $r$-th constraint corresponding to an information item causing inconsistency is eliminated. Thus, the aim is to minimize the sum of all ${\nu_r}, \ r = 1, \cdots ,l$, subject to the modified constraints  ${W^R}$. The specific mixed $\{0,1\}$ linear programming problem is as follows:

\begin{align*}
\quad & \text{(Model 4)}\hspace{0.1cm} \min \sum\limits_{r = 1}^l {{\nu _r}} \\
\quad & \text{\textup{subject to}}
\begin{array}{r@{~}l}
\begin{cases}
{W^{{A^R}}}^\prime
\begin{cases}
\hspace{0.1cm}{u_j}({a_s}) - {u_j}({{a_t}}) + M{\nu _r} \ge \varepsilon , \mbox{ if } {a_s}{ \succ ^{{d_j}}}{a_t}
\hspace{0.1cm} \text { for } {a_s},{a_t}\in {A^R},{d_j} \in D,\\[10pt]
\mbox{ if } {a_s}{ \sim ^{{d_j}}}{a_t}\hspace{0.1cm} \text { for } {a_s},{a_t}\in {A^R}, \ {d_j} \in D
\begin{cases}
\hspace{0.1cm}{u_j}({a_s}) - {u_j}({a_t}) + M{\nu _r} \ge 0,\\[10pt]
\hspace{0.1cm}{u_j}({a_t}) - {u_j}({a_s}) + M{\nu _r} \ge 0,
\end{cases}\\[10pt]
\hspace{0.1cm}{u_j}({a_s}) - {u_j}({a_t}) + M{\nu _r} \ge {u_j}({{a_p}}) - {u_j}({a_q}) + \varepsilon ,
\mbox{ if } ({a_s},{a_t}){ \succ ^{*{d_j}}}({a_p},{a_q})\\[10pt]\hspace{0.1cm} \text { for } {a_s},{a_t},{a_p},{a_q}\in {A^R}, \ {d_j} \in D,\vspace{2ex}\\
\mbox{ if } ({{a_s},{a_t}}){ \sim ^{*{d_j}}}({{a_p},{a_q}})
\begin{cases}
\hspace{0.1cm} ({u_j}({a_s}) - {u_j}({a_t}))-({u_j}({a_p}) - {u_j}({a_q}))+ M{\nu _r} \ge 0\\
\hspace{0.1cm}  ({u_j}({a_p}) - {u_j}({a_q}))-({u_j}({{a_s}}) - {u_j}({a_t}))+ M{\nu _r} \ge 0
\end{cases}\vspace{2ex}\\[10pt]
\hspace{0.1cm} \text {for } {a_s},{a_t},{a_p},{a_q}\in {A^R}, \ {d_j} \in D,
\end{cases}\\[10pt]
{W^{{N^R}}}^\prime
\begin{cases}
 \ {w_{ij}} = 0\hspace{0.1cm} \text { for } i,j\in {N^R}
\begin{cases}
\hspace{0.1cm}{w_{ij}} + M{\nu _r} \ge 0,\\[10pt]
\hspace{0.1cm} - {w_{ij}} + M{\nu _r} \ge 0,
\end{cases}\\[10pt]
\hspace{0.1cm}{w_{ij}}-{w_{kh}}  + M{\nu _r} \ge \varepsilon \hspace{0.1cm} \text { for } i,j,k,h\in {N^R},\\[10pt]
\hspace{0.1cm}{w_{ij}} - \frac{1}{2}{\omega _j} + M{\nu _r} \ge 0 \hspace{0.1cm} \text { for } i,j\in {N^R},\\[10pt]
\hspace{0.1cm}{\omega _i} - {\omega _j} + M{\nu _r} \ge {\omega _k} - {\omega _h} + \varepsilon \hspace{0.1cm} \text { for } i,j,k,h\in {N^R},
\end{cases}\\[10pt]
{\nu _r} \in \{ 0,1\} \hspace{0.1cm}, \ r = 1,2, \cdots,l,\\[10pt]
{W^{base},}
\end{cases}
\end{array}
\end{align*}
where $M$ is a positive large number.

The optimal solution of Model 4 gives one of possibly many sets of inconsistent constraints with the smallest cardinality. It is composed of constraints getting ${v_r} = 1$. Other sets of inconsistent constraints can be obtained by solving Model 4 with the additional constraint $\sum\limits_{r \in T} {{v_r} \le |T|}  - 1$, where $T$ is the set of constraints corresponding to ${v_r} = 1$ in the last optimal solution of Model 4. This additional constraint prevents Model 4 from finding the same solution again. Finally, the subsets of inconsistent constraints are presented to the DMs who may select which one should be eliminated from the preference and network node information to ensure feasibility of ${W^R}$.

\subsection{Selection of a representative empathic network}
The feasible region of ${W^R}$ defines the range of empathic networks that are compatible with the preference and network node information of the DMs, which ensures the robustness of the results. In some realistic situations, DMs have some basic judgments about network structure and need to get clear-cut results. Therefore, the problem addressed in this subsection is how to select the most representative empathic network from the set of compatible empathic networks $\mathcal{W}$ in combination with the personalized needs of the DMs.

\subsubsection{The most discriminating empathic network}
Consider that the DMs want to get a single empathic network to make the provided decision information as clear as possible. We propose  the following four kinds of interpretation of the preference and network node information in view of obtaining a representative compatible empatic network:
\vspace{0.3cm}

(1)	If all ${d_j}\in D$ firmly express their preference relations between alternatives, such as ${a_s}{ \succ ^{{d_j}}}{a_t}, \ {a_s},{a_t}\in {A^R}$, they hope the difference in empathic utility between alternatives is as large as possible.
\vspace{0.3cm}

(2)	If the intensities of preference relations between alternatives are firmly expressed  by all ${d_j} \in D$, such as $( {{a_s},{a_t}}){ \succ ^{*{d_j}}}({{a_p},{a_q}}), \ {a_s},{a_t},{a_p},{a_q}\in {A^R}$,  they  want the differences $({u_j}({{a_s}}) - {u_j}({{a_t}})) - ({u_j}({{a_p}}) - {u_j}({{a_q}}))$ be as large as possible.
\vspace{0.3cm}

(3)	If the DMs want the intensities of the empathic relationships between nodes, such as ${w_{ij}}>{w_{kh}}, \ i,j,k,h\in {N^R}$, to be amplified, then the weight differences ${w_{ij}}-{w_{kh}}$ should  be as large as possible.
\vspace{0.3cm}

(4)	If the DMs want the intensities of the empathic centrality relationships between nodes, such as ${\omega _i} - {\omega _j} > {\omega _k} - {\omega _h}, \ i,j,k,h\in {N^R}$, to be amplified, then the differences of empatic centralities  $(({\omega _i} - {\omega _j}) - ({\omega _k} - {\omega _h}))$ should be as large as possible.
\vspace{0.3cm}

The modeling of the above four kinds of interpretation of the preference and network node information consists in finding the maximum value of $\varepsilon$ for the following four types of constraints included in set ${W^R}$:
\begin{align*}
\begin{array}{l@{~}r}
{u_j}({{a_s}}) \ge {u_j}({{a_t}}) + \varepsilon , \mbox{ if } {a_s}{ \succ ^{{d_j}}}{a_t}  \hspace{0.1cm} \text { for } {a_s},{a_t}\in {A^R}, \ {d_j} \in D,\\
{u_j}({{a_s}}) - {u_j}({{a_t}}) \ge {u_j}({{a_p}}) - {u_j}({{a_q}}) + \varepsilon , \mbox{ if } ({{a_s},{a_t}}){ \succ ^{*{d_j}}}({{a_p},{a_q}})\\
\hspace{0.1cm} \text { for } {a_s},{a_t},{a_p},{a_q}\in {A^R}, \ {d_j} \in D,\\
{w_{ij}} \ge {w_{kh}} + \varepsilon \hspace{0.1cm}  \text { for } i,j,k,h\in {N^R},\\
{\omega _i} - {\omega _j} \ge {\omega _k} - {\omega _h} + \varepsilon \hspace{0.1cm}  \text { for }  i,j,k,h\in {N^R}.\\
\end{array}
\end{align*}

The goal of getting the most discriminating compatible empatic network can be achieved by solving the following model:

\begin{align*}
\begin{array}{r@{~}l}
\quad &\text{(Model 5)}\hspace{0.1cm}\max \varepsilon \\[10pt]
\quad &\text{\textup{subject to}}
\hspace{0.1cm} {W^R}.\\[10pt]
\end{array}
\end{align*}

The optimal solution of Model 5 emphasizes the firmly expressed relations between the reference alternatives, and between the weights and empathic centralities of some individuals. This  operation can reduce the ambiguity of recommendation of the final ranking.

\subsubsection{The sparse empathic network}
How to train a sparse neural network that performs as well as a dense neural network? This is a hot issue in the field of deep learning. Moving this question to the search of a representative empathic network, we are looking for the sparsest compatible empathic network. The significance lies in the ability of sparse networks to facilitate rapid information circulation, enhance network efficiency, while maintaining robustness.

In order to select the sparsest empathic network from the set of compatible empathic networks $\mathcal{W}$, a new constraint $\varepsilon^\prime {\gamma _{ij}} \le {w _{ij}} \le M{\gamma _{ij}}$ is added, where $\varepsilon^\prime$ is the empathic relationship intensity threshold, $\gamma_{ij}$ is a $\{0,1\}$ variable, and $M$ is a positive large number. When $\gamma _{ij}=0$, then $w_{ij}=0$, which means that ${d_i}$ has no empathic relationship with ${d_j}$, while when $\gamma _{ij}=1$, then $w_{ij}\ge \varepsilon^\prime$, indicating that ${d_i}$ has an empathic relationship with ${d_j}$. Minimizing the sum of all ${\gamma _{ij}}, \ i,j\in N$, as much as possible will result in a sparse configuration of the network. Therefore, in order to obtain a sparse compatible empathic network one has to solve the following Model 6:

\begin{align*}
\begin{array}{r@{~}l}
\quad &\text{(Model 6)}\hspace{0.1cm}\min \sum\limits_i {\sum\limits_j {{\gamma _{ij}}} }  \\[10pt]
\quad &\text{\textup{subject to}}
\begin{cases}
\hspace{0.1cm} {W^R},\\[10pt]
\hspace{0.1cm} \varepsilon^\prime {\gamma_{ij}} \le {w_{ij}} \le M{\gamma_{ij}}, \ \forall i,j\in N,\\[10pt]
\hspace{0.1cm} {\gamma _{ij}} \in \{ 0,1\}, \ \forall i,j\in N,
\end{cases}
\end{array}
\end{align*}
where $\varepsilon^\prime$ is the empathic relationship intensity threshold and $M$ is a positive large number.

\subsubsection{Central and distributed empathic networks}
In a social network, a node possessing significant influence means it is the ``opinion leader" in group decision-making, with other nodes acting as its followers. This phenomenon is commonly observed in real-life scenarios, such as individuals with notable talent in social circles, critical infrastructures within urban networks, or pivotal entities in disease transmission networks. This type of network is generally referred to as a central social network and is characterized by high efficiency and low resilience. Sometimes we need, however, to avoid the situation where the ``opinion leaders'' are attacked and the whole network is brought down. Therefore, the equal status of all nodes in the network may be required, which means that their influence is balanced. Such networks are often referred to as the distributed social networks. Based on this, this subsection introduces two types of empathic networks with different influence characteristics, i.e., the central empathic network and the distributed empathic network.

In order to construct the models permitting to select central and distributed empathic networks from the set of compatible empathic networks $\mathcal{W}$, we first give the definition of entropy based on the empathic centrality. Then, the definitions of central and distributed empathic networks are provided. Finally, the models are formulated according to the goal of maximizing or minimizing the entropy.

\begin{definition}
Given the empathic network $G(N,E)$, the empathic centrality $\omega _j$ of ${d_j}\in D$ is normalized to $\ \omega _j^\prime= \frac{{{\omega _j}}}{n},\  j \in N$. Then,
\begin{align}
-\sum\limits_{j = 1}^n {\omega _j^\prime} \ln \omega _j^\prime,
\end{align}
is the entropy based on the empathic centrality.
\end{definition}

Note that since ${w_{ij}} \ge 0, i \ne j$ and ${w_{jj}} > 0,j = 1,2, \cdots, n$, we have ${\omega _j} = \sum\limits_{i = 1}^n {{w_{ij}}}  \ge {w_{jj}} > 0$  and consequently $\omega _j^\prime = \frac{{{\omega _j}}}{n} > 0$.

\begin{theorem}\label{T1}
When the empathic centrality of all ${d_j}\in D$ in the empathic network is evenly distributed, that is, the status of nodes is equal, the entropy based on the empathic centrality reaches the maximum value. When the empathic centrality of  ${d_k}\in D$ is the largest, that is, the node is the ``opinion leader" of the network, the entropy reaches the minimum value.
\end{theorem}
The proof is given in Appendix A.

\begin{definition}\label{dy1}
If the empathic centrality of ${d_j}\in D$ in the empathic network is large enough, that is, ${\omega _j}\ge\frac{n}{2}, \ j\in N$, is satisfied, then the empathic network is called the central empathic network.
\end{definition}

\begin{definition}\label{dy2}
If there exists $\delta> 0$ such that the empathic centrality $\omega _j$ of all ${d_j}\in D$ in the empathic network satisfies $|{\omega_j}-1| \le \delta,\ j \in N$, then such an empathic network is called the distributed empathic network.
\end{definition}

It is worth noting that the condition $|{\omega _j} - 1| \le \delta, \ j \in N$, in Definition~\ref{dy2} means that the empathic centrality $\omega_j$ of ${d_j}\in D$ in the network is as close as possible to their mean value 1. According to Theorem~\ref{T1}, all nodes in the network have equal status, thus it is the distributed empathic network.

In the following, we formulate the two models for finding the central empathic network and the distributed empathic network in the set of compatible empathic networks $\mathcal{W}$:

\begin{align*}
\begin{array}{r@{~}l}
\quad &\text{(Model 7)}\hspace{0.1cm}\min \ - \sum\limits_{j = 1}^n {\omega _j^\prime} \ln \omega _j^\prime  \\[10pt]
\quad &\text{\textup{subject to}}
\begin{cases}
\hspace{0.1cm} {W^R},\\[10pt]
\hspace{0.1cm} \omega _j^\prime = \frac{{{\omega _j}}}{n}, \ \forall j\in N.
\end{cases}
\end{array}
\end{align*}

\begin{align*}
\begin{array}{r@{~}l}
\quad &\text{(Model 8)}\hspace{0.1cm}\max \ - \sum\limits_{j = 1}^n {\omega _j^\prime} \ln \omega _j^\prime  \\[10pt]
\quad &\text{\textup{subject to}}
\begin{cases}
\hspace{0.1cm} {W^R},\\[10pt]
\hspace{0.1cm} \omega _j^\prime = \frac{{{\omega _j}}}{n},\  \forall j\in N.
\end{cases}
\end{array}
\end{align*}

Let $g(\omega _j^\prime) =  - \omega _j^\prime\ln \omega _j^\prime$. It has been confirmed in the proof of Theorem~\ref{T1} that $g(\omega _j^\prime)$ is a concave function on the domain $\omega _j^\prime \in [0,1]$, so $\sum\limits_{j = 1}^n {\omega _j^\prime} \ln \omega _j^\prime$ is a convex function on $\omega _j^\prime \in [0,1]$. The constraints of Model 7 and Model 8 are linear constraints. Therefore, Model 7 and Model 8 are nonlinear programming problems whose objective functions are to find the maximum and the minimum value of $\sum\limits_{j = 1}^n {\omega _j^\prime} \ln \omega _j^\prime$, subject to the linear constraints, which can be solved by Lingo 18.0.

\subsubsection{The resilient empathic network}
Life is often marked by various unpredictable disturbances, ranging from minor occurrences like a cold or fever to more significant challenges such as infrastructure paralysis or even existential threats. Given the negative effects of disruptions, there is an increasing focus on system resilience. Resilience \cite{14C.S.Holling1973} is widely used in ecology, finance, transportation and other fields, and can be understood as the response of a system in face of disturbance events, or the ability of a system to adjust its behavioral activities to maintain basic functions in case of disturbances. Considering that expert withdrawal is likely to occur in the decision-making process of emergency management, in this subsection, we study local and global empathic networks with high resilience.

The usual approach to improving network resilience involves two aspects. The first is to reduce the importance of key nodes and make the organization more flat. The second is to increase redundancy while keeping the functionally identical network nodes and links  \cite{15R.S.Burt1982}. In order to make a quantitative analysis, we introduce the directed network density $\rho$, and then give a definition of the high-resilient empathic network.

\begin{definition}
For an empathic network containing $N$ nodes, the density $\rho$ is defined as the ratio of the actual number of edges $N^\prime$ (excluding self-loops) in the network to the maximum possible number of edges, that is, $\rho =\frac{N^\prime}{{N(N - 1)}}$.
\end{definition}

\begin{definition}\label{d}
If the empathic network density $\rho \ge {\rho _0}$ (the threshold given in advance) and there is $\delta  > 0$, so that the empathic centrality of each node in the empathic network satisfy $|{\omega _j} - 1| \le \delta, \ \forall j\in N$, then the empathic network is a high-resilient empathic network.
\end{definition}

In order to select the local empathic network with high resilience from the set of compatible empathic networks $\mathcal{W}$, we formulate Model 9, whose role is to make the empathic centrality of each node as equal as possible under the condition that there are as many links in the network as possible. Model 9 can be written as a convex programming problem in which the objective function is to minimize the convex function $\sum\limits_{j = 1}^n {\omega _j^\prime } \ln \omega _j^\prime$ subject to linear constraints. One can use again Lingo 18.0 to solve it.

\begin{align*}
\begin{array}{r@{~}l}
\quad &\text{(Model 9)}\hspace{0.1cm}\max \ - \sum\limits_{j = 1}^n {\omega _j^\prime } \ln \omega _j^\prime  \\[10pt]
\quad &\text{\textup{subject to}}
\begin{cases}
\hspace{0.1cm} {W^{{A^R}}},\\[10pt]
\hspace{0.1cm} {W^{{N^R}}},\\[10pt]
\hspace{0.1cm} \sum\limits_{k = 1}^{n} {{w_{jk}} = 1}, \ \forall j\in N,\\[10pt]
\hspace{0.1cm} {w_{ij}} \ge {\varepsilon^\prime}, \ \forall i,j\in N,\\[10pt]
\hspace{0.1cm} \omega _j^\prime  = \frac{{{\omega _j}}}{n}, \ \forall j\in N.\\[10pt]
\end{cases}
\end{array}
\end{align*}

Given the transitive nature of empathic relationships, our consideration shifts towards ensuring that the global empathic network maintains high resilience after the transference of these relationships.

\begin{definition}\cite{36R.S.Varga1962}\label{d1}
Given $A \in {C^{n \times n}}, \ n \ge 2$. If there exists a permutation matrix $P$ such that

\begin{equation*}
PA{P^T}=\begin{pmatrix}
A_{11} & A_{12} \\
0& A_{22}\\
\end{pmatrix}, \ {A_{11}} \in {C^{r \times r}}, \ {A_{22}} \in {C^{(n - r) \times (n - r)}}
\end{equation*}
where $1 \le r \le n$, then $A$ is reducible, otherwise $A$ is irreducible.
\end{definition}

\begin{theorem}\label{T2}
The local empathic weight matrix $W$ is irreducible if and only if the global empathic weight matrix $G>0$.

\end{theorem}
The proof is given in Appendix B.

The global empathic weight matrix $G>0$ means that $G$ is a positive matrix with any of its entries being positive. This also implies that there are bidirectional links between any two nodes in the empathic network, so the global empathic network is a directed complete graph\footnote{In a directed graph, if there are two arcs with opposite directions between any two nodes, the graph is a directed complete graph.}.

\begin{corollary}\label{C1}
Local empathic weight matrix $W$ corresponding empathic network has a loop, it passes at least once through each node if and only if the global empathic weight matrix $G>0$.
\end{corollary}

\begin{corollary}\label{C2}
If local empathic weight matrix $W$ has the form:

\vspace{0.3cm}
\begin{equation*}
\begin{pmatrix}
w_{11} & w_{12}&  & &  \\
 & w_{22} & w_{23}& & \\
  & &\ddots &\ddots &  \\
 & & &w_{n-1,n-1}&w_{n-1,n}\\
 w_{n1}& & & &w_{n,n}\\
\end{pmatrix} or
\begin{pmatrix}
 w_{11}& & & &w_{1,n} \\
 w_{21}& w_{22} & & & \\
 & \ddots &\ddots &  & \\
 & &w_{n-1,n-2}&w_{n-1,n-1}& \\
 & & &w_{n,n-1}&w_{n,n}\\
\end{pmatrix},
\end{equation*}
\vspace{0.3cm}

\noindent where the blank position means that any non-negative real number can be there, then the global empathic weight matrix $G>0$.
\end{corollary}

From the above theorems and corollaries, we know when the network corresponding to local empathic matrix $W$ is irreducible, and the global empathic network obtained after the transference of empathic relationships has a high network density. In order to form a highly resilient global empathic network, we also need to further consider the altruism of the empathic relationship, that is, DMs should have an open attitude and pay more attention to the utility or preference of other adjacent nodes. So it is also necessary to ensure that the self-empathic weight $w_{jj}$ of each DM is small enough. Therefore, we can build Model 10 and Model 11 to select the local empathic network from the set of compatible empathic networks $\mathcal{W}$.

\begin{align*}
\begin{array}{r@{~}l}
\quad &\text{(Model 10)}\hspace{0.1cm}\max \ {\varepsilon}  \\[10pt]
\quad &\text{\textup{subject to}}
\begin{cases}
\hspace{0.1cm} {W^{{A^R}}},\\[10pt]
\hspace{0.1cm} {W^{{N^R}}},\\[10pt]
\hspace{0.1cm} \sum\limits_{k = 1}^{n} {{w_{jk}} = 1}, \ \forall j\in N,\\[10pt]
\hspace{0.1cm} {w_{ij}} \ge {\varepsilon^\prime}, \ \ i \ne j,i,j \in N,\\[10pt]
\hspace{0.1cm} {w_{k,k + 1}} \ge {\varepsilon^\prime} \hspace{0.1cm} \text { for } k = 1,2, \cdots ,n - 1,(10-1)\\[10pt]
\hspace{0.1cm} {w_{n1}} \ge {\varepsilon^\prime},(10-2)\\[10pt]
\hspace{0.1cm} {w_{jj}} = {\varepsilon^\prime},j \in N.(10-3)
\end{cases}
\end{array}
\end{align*}

\begin{align*}
\begin{array}{r@{~}l}
\quad &\text{(Model 11)}\hspace{0.1cm}\max \ {\varepsilon} \\[10pt]
\quad &\text{\textup{subject to}}
\begin{cases}
\hspace{0.1cm} {W^{{A^R}}},\\[10pt]
\hspace{0.1cm} {W^{{N^R}}},\\[10pt]
\hspace{0.1cm} \sum\limits_{k = 1}^{n} {{w_{jk}} = 1}, \ \forall j\in N,\\[10pt]
\hspace{0.1cm} {w_{ij}} \ge {\varepsilon^\prime}, \ \ i \ne j,i,j \in N,\\[10pt]
\hspace{0.1cm} {w_{k,k- 1}} \ge {\varepsilon^\prime} \hspace{0.1cm} \text { for } k = 2,3, \cdots ,n,\\[10pt]
\hspace{0.1cm} {w_{1n}} \ge {\varepsilon^\prime},\\[10pt]
\hspace{0.1cm} {w_{jj}} = {\varepsilon^\prime},j \in N
\end{cases}
\end{array}
\end{align*}

In Model 10, constraints (10-1) and (10-2) are to ensure that there is a loop in the network corresponding to local empathic weight matrix $W$, and constraint (10-3) is to ensure that the self-empathic weights of the nodes are low enough to pay attention to the preferences of other nodes. The corresponding constraints in Model 11 have the same meaning as in Model 10.

\subsubsection{Empathic networks with different topologies}
In recent years, social networks have attracted increasing attention in the fields of psychology, neuroscience and computing. Related studies have found that the topological structure of social networks shapes collective cognition and behavior \cite{37A.Coman2016,38I.Momennejad2019,39D.Centola2010,40I.Momennejad2022}. Human cognition is not isolated, but is formed in social networks through learning and memory. Unlike animals, human social networks have different topologies and serve different environments. Humans transmit and share information through social networks to synchronize collective cognition and memory and integrate different information and knowledge. Therefore, in this subsection, we consider the influence of different empathic network topologies on collective cognition, and we construct Models 12-14 to select compatible empathic networks that can not only reproduce the decision information provided by the DMs, but also meet the requirements of different network topologies.

In emergency decision-making, a star topology can quickly centralize information, aiding in rapid response. A bus topology can simplify the pathways for information flow, reducing delays and thereby increasing efficiency. A tree topology can provide experts with a clear framework for information circulation, facilitating the swift conveyance of commands and feedback. Let's consider star, bus, and tree network topologies. To represent them, we will add specific constraints to $W^{R}$ in the new models.

\begin{figure}[!h]
\centering
\subfigure[The star topology]{
\resizebox*{3cm}{!}{\includegraphics{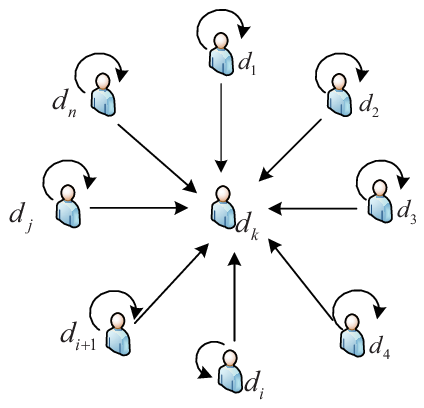}}\label{2-1}}
\subfigure[The bus topology]{
\hspace{0.2cm}\resizebox*{4cm}{!}{\includegraphics{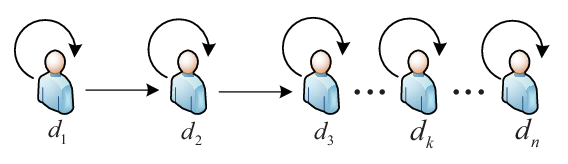}}\label{2-2}}
\subfigure[The tree topology]{
\hspace{0.2cm}\resizebox*{4cm}{!}{\includegraphics{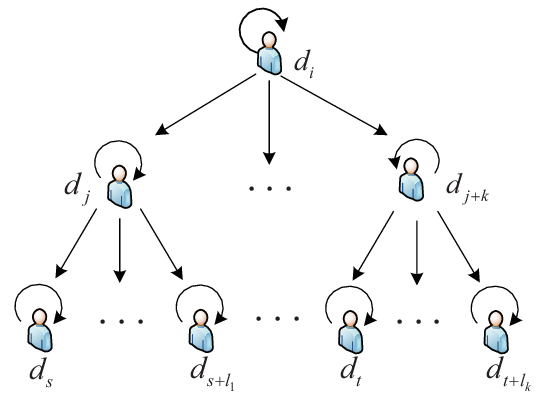}}\label{2-3}}
\caption{Three topologies of empathic networks.}
\end{figure}

\begin{itemize}
\item A star topology is a center with multiple nodes. Each sub-node has an empathic relationship with the central node, as shown in Fig.~\ref{2-1}. Model 12 permits to select the compatible empathic network with star topology.
\begin{align*}
\begin{array}{r@{~}l}
\quad &\text{(Model 12)}\hspace{0.1cm}\max\  {\varepsilon} \\[10pt]
\quad &\text{\textup{subject to}}
\begin{cases}
\hspace{0.1cm} {W^{^R}},\\[10pt]
\hspace{0.1cm} {w_{ik}} \ge {\varepsilon ^\prime} \hspace{0.1cm} \text { for } i, k \in N \hspace{0.1cm} \mbox{and} \hspace{0.1cm} i \ne k, \\[10pt]
\hspace{0.1cm} {w_{ij}}= 0 \hspace{0.1cm} \text { for } i, j \in N \hspace{0.1cm} \mbox{and} \hspace{0.1cm} i \ne j.
\end{cases}
\end{array}
\end{align*}
\end{itemize}

\begin{itemize}
\item A bus topology means that all nodes are linked through a single bus. As shown in Fig.~\ref{2-2}, nodes $1,2, \cdots ,n$ have the empathic relationships in turn.
\begin{align*}
\begin{array}{r@{~}l}
\quad &\text{(Model 13)}\hspace{0.1cm}\max \ {\varepsilon} \\[10pt]
\quad &\text{\textup{subject to}}
\begin{cases}
\hspace{0.1cm} {W^{^R}},\\[10pt]
\hspace{0.1cm} {w_{k,k+1}} \ge {\varepsilon ^\prime} \hspace{0.1cm} \text { for } k = 1,2, \cdots ,n - 1, \\[10pt]
\hspace{0.1cm} {w_{ij}}= 0 \hspace{0.1cm} \text { for } i, j \in N \hspace{0.1cm} \mbox{and} \hspace{0.1cm} i \ne j,i \ne j - 1.
\end{cases}
\end{array}
\end{align*}

In Model 13, constraints ${w_{k,k+1}} \ge {\varepsilon ^\prime}, \ k = 1,2, \cdots ,n - 1$, can be written to ${w_{k+1,k}} \ge {\varepsilon ^\prime}, \ k = 1,2, \cdots ,n - 1,$ representing the opposite bus links to Fig.~\ref{2-2}.
\end{itemize}

\begin{itemize}
\item The tree topology evolves from the bus topology and is shaped like an inverted tree with a root at the top and branches below the root, and each branch can be followed by sub-branches. Fig.~\ref{2-3} shows the empathic network with a three-layer tree structure. Model 14 permits to find a compatible empathic network with a tree topology.
\begin{align*}
\begin{array}{r@{~}l}
\quad &\text{(Model 14)}\hspace{0.1cm}\max \ {\varepsilon} \\[10pt]
\quad &\text{\textup{subject to}}
\begin{cases}
\hspace{0.1cm} {W^{^R}},\\[10pt]
\hspace{0.1cm} {w_{ij}} + {w_{i,j+1}} +\cdots + {w_{i,j+k}} = 1 \hspace{0.1cm} \text { for } i \ne j, \ k \in N,\\[10pt]
\hspace{0.1cm} {w_{ij}}, {w_{i,j+1}}, \cdots, {w_{i,j+k}}\ge {\varepsilon ^\prime}\hspace{0.1cm} \text { for } i \ne j, \ k \in N,\\[10pt]
\hspace{0.1cm} {w_{js}} + {w_{j,s+1}} +\cdots + {w_{j,s+l_{1}}} = 1 \hspace{0.1cm} \text { for } j \ne s, \ l_{1} \in N,\\[10pt]
\hspace{0.1cm} {w_{js}}, {w_{j,s+1}}, \cdots, {w_{j,s+l_{1}}}\ge {\varepsilon ^\prime} \hspace{0.1cm} \text { for } j \ne s, \ l_{1} \in N,\\[10pt]
\hspace{0.1cm} \cdots\\
\hspace{0.1cm} {w_{j+k,t}} + {w_{j+k,t+1}} +  \cdots  + {w_{j+k,t+l_{k}}} = 1 \hspace{0.1cm} \text { for } t\ne j+k, \ l_{k} \in N,\\[10pt]
\hspace{0.1cm} {w_{j+k,t}}, {w_{j+k,t+1}}, \cdots, {w_{j+k,t+l_{k}}}\ge {\varepsilon ^\prime}\hspace{0.1cm} \text { for } t\ne j+k, \ l_{k} \in N.
\end{cases}
\end{array}
\end{align*}

\end{itemize}

\section{Numerical example}
Suppose, during an emergency rescue operation in a specific location, the government invites ten experts $D=\{d_1, d_2, d_3,\cdots, d_{10}\}$ from different fields to collaborate on the construction of emergency shelters with the aim of providing suitable settlements for the victims. It is known that there are five kinds of building materials for experts to choose: lumber ($a_1$), bamboo ($a_2$), stone ($a_3$), cement ($a_4$), and concrete ($a_5$). In order to ensure the proper course of the decision-making process, the competent decision-making body should not only take into account the decision information provided by the experts, but also pay attention to the influence of existing empathic relationships between experts on decision-making results. Therefore, it is necessary to infer the empathic network among the experts to provide a credible reference for subsequent decision analysis. Then, the intrinsic utilities of experts are the calculated by the robust ordinal regression approach in subsection~\ref{intrinsic utilities}. The binary empathic relationships among experts are calculated in subsection~\ref{eryuan} using the empathic network learning method, and the results of selecting representative empathic networks from a set of compatible empathic networks under different requirements are shown in subsection~\ref{daibiao}. The decision results under different network structures are presented in subsection~\ref{decision}.

\subsection{Calculation of the intrinsic utilities}\label{intrinsic utilities}
In this subsection, we assume that experts can only provide incomplete fuzzy judgment matrices, so we cannot directly capture intrinsic utilities. Firstly, according to experts' own experience and knowledge, each of them gives a incomplete judgment matrix of the alternative materials, as follows:

{\tiny
\begin{equation*}
{R_1}=\begin{pmatrix}
0.5&0.8&*&*&*\\
0.2&0.5&*&*&*\\
*&*&0.5&0.7&0.4\\
*&*&0.3&0.5&0.2\\
*&*&0.6&0.8&0.5\\
\end{pmatrix};
{R_2}=\begin{pmatrix}
0.5&0.4&*&*&*\\
0.6&0.5&*&*&*\\
*&*&0.5&0.2&0.3\\
*&*&0.8&0.5&0.6\\
*&*&0.7&0.4&0.5\\
\end{pmatrix};
{R_3}=\begin{pmatrix}
0.5&0.3&0.6&*&*\\
0.7&0.5&0.8&*&*\\
0.4&0.2&0.5&*&*\\
*&*&*&0.5&0.2\\
*&*&*&0.8&0.5\\
\end{pmatrix};
{R_4}=\begin{pmatrix}
0.5&0.7&0.6&0.9&0.6\\
0.3&0.5&0.4&*&*\\
0.4&0.6&0.5&*&*\\
0.1&*&*&0.5&0.2\\
0.4&*&*&0.8&0.5\\
\end{pmatrix};
{R_5}=\begin{pmatrix}
0.5&0.1&*&*&*\\
0.9&0.5&*&*&*\\
*&*&0.5&0.6&0.8\\
*&*&0.4&0.5&0.7\\
*&*&0.2&0.3&0.5\\
\end{pmatrix};
\end{equation*}
\begin{equation*}
{R_6}=\begin{pmatrix}
0.5&0.2&*&*&0.3\\
0.8&0.5&*&*&0.6\\
*&*&0.5&0.7&0.5\\
*&*&0.3&0.5&0.3\\
0.7&0.4&0.6&0.8&0.5\\
\end{pmatrix};
{R_7}=\begin{pmatrix}
0.5&0.8&0.6&*&*\\
0.2&0.5&0.3&*&*\\
0.4&0.7&0.5&*&*\\
*&*&*&0.5&0.7\\
*&*&*&0.3&0.5\\
\end{pmatrix};
{R_8}=\begin{pmatrix}
0.5&0.4&*&*&0.6\\
0.6&0.5&*&*&0.7\\
*&*&0.5&0.8&0.5\\
*&*&0.2&0.5&0.2\\
0.4&0.3&0.5&0.8&0.5\\
\end{pmatrix};
{R_9}=\begin{pmatrix}
0.5&0.4&0.5&*&*\\
0.6&0.5&0.6&*&*\\
0.5&0.4&0.5&*&*\\
*&*&*&0.5&0.3\\
*&*&*&0.7&0.5\\
\end{pmatrix};
{R_{10}}=\begin{pmatrix}
0.5&0.7&*&*&*\\
0.3&0.5&*&*&*\\
*&*&0.5&0.8&0.3\\
*&*&0.2&0.5&0\\
*&*&0.7&1&0.5\\
\end{pmatrix}.
\end{equation*}}

In addition, experts provide some indirect preference and preference intensity information, as follows:
\vspace{0.3cm}

${a_1}\succeq^{1}{a_3}$,\ ${a_2}\succeq^{1}{a_4}$, \ $({a_1},{a_4}) \succeq^{*1}({a_1},{a_3})$;
${a_2}\succeq^{2}{a_3}$,\ ${a_2}\succeq^{2}{a_5}$, \ $({a_1},{a_5}) \succeq^{*2}({a_1},{a_4})$;\\

${a_3}\succeq^{3}{a_4}$, ${a_5}\succeq^{3}{a_3}$, \ $({a_1},{a_4}) \succeq^{*3}({a_2},{a_5})$;
${a_3}\succeq^{4}{a_4}$, \ $({a_2},{a_4}) \succeq^{*4}({a_2},{a_5})$;\\

${a_2}\succeq^{5}{a_3}$,\ ${a_2}\succeq^{5}{a_4}$,\ $({a_2},{a_5}) \succeq^{*5}({a_1},{a_3})$;
${a_2}\succeq^{6}{a_4}$,\ $({a_2},{a_3}) \succeq^{*6}({a_1},{a_4})$;\\

${a_4}\succeq^{7}{a_3}$,\ $({a_2},{a_5}) \succeq^{*7}({a_2},{a_4})$;
${a_1}\succeq^{8}{a_3}$,\ $({a_2},{a_4}) \succeq^{*8}({a_1},{a_4})$;\\

${a_1}\succeq^{9}{a_4}$, \ ${a_5}\succeq^{9}{a_3}$,\ $({a_2},{a_4}) \succeq^{*9}({a_1},{a_5})$;
${a_1}\succeq^{10}{a_3}$, \ ${a_2}\succeq^{10}{a_4}$,\ $({a_1},{a_4}) \succeq^{*10}({a_1},{a_3})$.

\vspace{0.3cm}
According to the above incomplete preference information provided by the experts, their sets of reference alternatives are: $A_1^R=A_6^R=A_8^R=A_{10}^R=\{{a_1},{a_2}, {a_3},{a_4}\}$, $A_2^R=A_3^R=A_5^R=A_9^R=\{{a_1},{a_2}, {a_3}, {a_4},{a_5}\}$, $A_4^R=A_7^R=\{{a_2},{a_3}, {a_4}, {a_5}\}$. Based on the above indirect preference information, matrices $R_1-R_5$ are completed using Model 1 as follows:
{\tiny
\begin{equation*}
{R_1}=\begin{pmatrix}
0.5&0.8&0.8&1&0.7\\
0.2&0.5&0.5&0.7&0.4\\
0.2&0.5&0.5&0.7&0.4\\
0&0.3&0.3&0.5&0.2\\
0.3&0.6&0.6&0.8&0.5\\
\end{pmatrix};
{R_2}=\begin{pmatrix}
0.5&0.4&0.9&0.6&0.7\\
0.6&0.5&1&0.7&0.8\\
0.1&0&0.5&0.2&0.3\\
0.4&0.3&0.8&0.5&0.6\\
0.3&0.2&0.7&0.4&0.5\\
\end{pmatrix};
{R_3}=\begin{pmatrix}
0.5&0.3&0.6&0.7&0.4\\
0.7&0.5&0.8&0.9&0.6\\
0.4&0.2&0.5&0.6&0.3\\
0.3&0.1&0.4&0.5&0.2\\
0.6&0.4&0.7&0.8&0.5\\
\end{pmatrix};
{R_4}=\begin{pmatrix}
0.5&0.7&0.6&0.9&0.6\\
0.3&0.5&0.4&0.7&0.4\\
0.4&0.6&0.5&0.8&0.5\\
0.1&0.3&0.2&0.5&0.2\\
0.4&0.6&0.5&0.8&0.5\\
\end{pmatrix};
{R_5}=\begin{pmatrix}
0.5&0.1&0.3&0.4&0.6\\
0.9&0.5&0.7&0.8&1\\
0.7&0.3&0.5&0.6&0.8\\
0.6&0.2&0.4&0.5&0.7\\
0.4&0&0.2&0.3&0.5\\
\end{pmatrix};
\end{equation*}
\begin{equation*}
{R_6}=\begin{pmatrix}
0.5&0.2&0.3&0.5&0.3\\
0.8&0.5&0.6&0.8&0.6\\
0.7&0.4&0.5&0.7&0.5\\
0.5&0.2&0.3&0.5&0.3\\
0.7&0.4&0.6&0.8&0.5\\
\end{pmatrix};
{R_7}=\begin{pmatrix}
0.5&0.8&0.6&0.3&0.5\\
0.2&0.5&0.3&0&0.2\\
0.4&0.7&0.5&0.2&0.4\\
0.7&1&0.8&0.5&0.7\\
0.5&0.8&0.6&0.3&0.5\\
\end{pmatrix};
{R_8}=\begin{pmatrix}
0.5&0.4&0.6&0.9&0.6\\
0.6&0.5&0.7&1&0.7\\
0.4&0.3&0.5&0.8&0.5\\
0.1&0&0.2&0.5&0.2\\
0.4&0.3&0.5&0.8&0.5\\
\end{pmatrix};
{R_9}=\begin{pmatrix}
0.5&0.4&0.5&0.6&0.4\\
0.6&0.5&0.6&0.7&0.5\\
0.5&0.4&0.5&0.6&0.4\\
0.4&0.3&0.4&0.5&0.3\\
0.6&0.5&0.6&0.7&0.5\\
\end{pmatrix};
{R_{10}}=\begin{pmatrix}
0.5&0.7&0.7&1&0.5\\
0.3&0.5&0.5&0.8&0.3\\
0.3&0.5&0.5&0.8&0.3\\
0&0.2&0.2&0.5&0\\
0.5&0.7&0.7&1&0.5\\
\end{pmatrix}.
\end{equation*}}

The eigenvector method is used to obtain the eigenvector corresponding to the maximum eigenvalue of each judgment matrix, which shows intrinsic utilities of each expert about all the alternatives. Taking $R_1$ as an example, by solving $|\lambda I - {R_1}| = 0$, where $\lambda$ is the eigenvalue, and $I$ is the $5 \times 5$  identity matrix, the maximum eigenvalue of ${R_1}$ is obtained as ${\lambda _{\max }} = 2.2000$. According to equation (9), the eigenvector
${\bf{w}}={(0.6674,0.3814,0.3814,0.1907,0.4767)^T}$ corresponding to ${\lambda _{\max }} = 2.2000$ is obtained by solving $(2.2000I - {R_1})\bf{w} = 0$. Since the elements in the eigenvector can be, in general, arbitrarily large, and the utility values lie within the range [0,1], we apply normalization to scale these elements into [0,1], thereby enhances the interpretability of the results. Therefore, we have derived the intrinsic utilities of $d_1$ for alternatives as $(0.3182,0.1818,0.1818,0.0909,0.2273)^T$. Similarly, take the normalized eigenvectors corresponding to the maximum eigenvalues of ${R_2}$ to ${R_{10}}$ as the intrinsic utilities of experts $D = \{ {d_2},{d_3},\cdots, {d_{10}}\}$ with respect to alternatives $A = \{ {a_1},{a_2},{a_3},{a_4},{a_5}\}$, and the intrinsic utility matrix ${U^I}$ is calculated as follows:

\begin{equation*}
{U^I}=\begin{pmatrix}
 0.3182&0.1818&0.1818&0.0909&0.2273\\
 0.2556&0.3020&0.0702&0.2093&0.1629\\
 0.2000&0.2877&0.1561&0.1123&0.2438\\
 0.2696&0.1826&0.2261&0.0957&0.2261\\
 0.1444&0.3298&0.2371&0.1907&0.0980\\
 0.1403&0.2682&0.2256&0.1403&0.2256\\
 0.2182&0.0818&0.1727&0.3091&0.2182\\
 0.2459&0.2918&0.2000&0.0623&0.2000\\
 0.1918&0.2327&0.1918&0.1509&0.2327\\
 0.2857&0.1905&0.1905&0.0476&0.2857\\
\end{pmatrix}.
\end{equation*}

Further, the individual utilities of all experts is made public, and the preferences of experts regarding building materials may change due to the influence of empathic relationships. Due to limitations and heterogeneity of expert knowledge, only partial preorder of building materials and partial node information about the empathic network can be provided, as shown below:
\vspace{0.3cm}

(a) $ {a_1}\succeq^{d_1}{a_4}$; \ (b) ${a_1}\succeq^{d_2}{a_3}$; \ (c) ${a_2}\succeq^{d_5}{a_3}$; \ (d) ${a_5}\succeq^{d_9}{a_4}$.

(e) ${w_{23}} \ge \varepsilon^\prime$; \ (f) ${w_{11}} \ge 2{w_{13}}$.

According to the incomplete information (a)-(f) provided by the experts, the set of reference alternatives ${A^R} = \mathop  \cup \limits_{j \in N} A_j^R=A=\{{a_1},{a_2}, \cdots ,{a_5}\}$, and the set of reference nodes ${N^R}=\{1,2,3\}$.

\subsection{Calculation of the binary empathic relationships}\label{eryuan}
Based on the incomplete decision information (a)-(f), we can write constraints ${W^R}$:
\begin{align*}
\quad & {W^{R}}
\begin{array}{r@{~}l}
\begin{cases}
{W^{{A^R}}}
\begin{cases}
\hspace{0.1cm}0.3182w_{11}+0.2556w_{12}+0.2w_{13}+0.2696w_{14}+0.1444w_{15}\\
+0.1403w_{16}+0.2182w_{17}+0.2459w_{18}+0.1918w_{19}+0.2857w_{1,10}\\
\hspace{0.1cm}\ge
0.0909w_{11}+0.2093w_{12}+0.1123w_{13}+0.0957w_{14}+0.1907w_{15}\\
+0.1403w_{16}+0.3091w_{17}+0.0623w_{18}+0.1509w_{19}+0.0476w_{1,10} + \varepsilon ,\\
\hspace{0.1cm}0.3182w_{21}+0.2556w_{22}+0.2w_{23}+0.2696w_{24}+0.1444w_{25}\\
+0.1403w_{26}+0.2182w_{27}+0.2459w_{28}+0.1918w_{29}+0.2857w_{2,10}\\
\hspace{0.1cm}\ge 0.1818w_{21}+0.0702w_{22}+0.1561w_{23}+0.2261w_{24}+0.2371w_{25}\\
+0.2256w_{26}+0.1727w_{27}+0.2w_{28}+0.1918w_{29}+0.1905w_{2,10} + \varepsilon ,\\
\hspace{0.1cm}0.1818w_{51}+0.3020w_{52}+0.2877w_{53}+0.1826w_{54}+0.3298w_{55}\\
+0.2682w_{56}+0.0818w_{57}+0.2918w_{58}+0.2327w_{59}+0.1905w_{5,10}\\
\hspace{0.1cm}\ge 0.1818w_{51}+0.0702w_{52}+0.1561w_{53}+0.2261w_{54}+0.2371w_{55}\\
+0.2256w_{56}+0.1727w_{57}+0.2w_{58}+0.1918w_{59}+0.1905w_{5,10}+\varepsilon ,\\
\hspace{0.1cm}0.2273w_{91}+0.1629w_{92}+0.2438w_{93}+0.2261w_{94}+0.0980w_{95}\\
+0.2256w_{96}+0.2182w_{97}+0.2w_{98}+0.2327w_{99}+0.2857w_{9,10}\\
\hspace{0.1cm}\ge 0.0909w_{91}+0.2093w_{92}+0.1123w_{93}+0.0957w_{94}+0.1907w_{95}\\
+0.1403w_{96}+0.3091w_{97}+0.0623w_{98}+0.1509w_{99}+0.0476w_{9,10}+ \varepsilon ,
\end{cases}\\
{W^{{N^R}}}
\begin{cases}
\hspace{0.1cm}{w_{23}} \ge \varepsilon^\prime,\\[5pt]
\hspace{0.1cm}{w_{11}} \ge 2{w_{13}},
\end{cases}\\
{W^{base}}
\begin{cases}
{w_{11}} + {w_{12}} + {w_{13}} + {w_{14}} + {w_{15}}+{w_{16}} + {w_{17}} + {w_{18}} + {w_{19}} + {w_{1,10}} = 1,\\
{w_{21}} + {w_{22}} + {w_{23}} + {w_{24}} + {w_{25}}+{w_{16}} + {w_{17}} + {w_{18}} + {w_{19}} + {w_{1,10}} = 1,\\
{w_{31}} + {w_{32}} + {w_{33}} + {w_{34}} + {w_{35}}+{w_{16}} + {w_{17}} + {w_{18}} + {w_{19}} + {w_{1,10}} = 1,\\
\cdots\\
{w_{91}} + {w_{92}} + {w_{93}} + {w_{94}} + {w_{95}}+{w_{96}} + {w_{97}} + {w_{98}} + {w_{99}} + {w_{9,10}} = 1,\\
{w_{10,1}} + {w_{10,2}} + {w_{10,3}} + {w_{10,4}} + {w_{10,5}}+{w_{10,6}} + {w_{10,7}} + {w_{10,8}} + {w_{10,9}} + {w_{10,10}}= 1,\\
{w_{ij}} \ge 0, \ i \ne j \hspace{0.1cm} \mbox{and} \hspace{0.1cm} \forall i,j \in \{ 1,2,3,4,5\} ,\\
{w_{jj}} \ge \varepsilon^\prime, \ j = 1,2,\cdots,10,
\end{cases}\\
\end{cases}
\end{array}
\end{align*}
where $\varepsilon^\prime$ is the empathic relationship intensity threshold. All the empathic networks satisfying ${W^R}$ constitute the set of compatible empathic networks $\mathcal{W}$.

Taking ${d_1}$ and ${d_2}$ as an example, we calculate the binary empathic relationships between them using Model 2 and Model 3. Fixing the empathic relationship intensity threshold $\varepsilon^\prime$ at 0.01, we obtain that the optimal solutions corresponding to Model 2 and Model 3 are both ${\varepsilon ^*} = 0.1840 > 0$, indicating that ${d_1}{P^l}{d_2}$.

\subsection{Comparative analysis of the representative empathic networks}\label{daibiao}
In this subsection, using Models 5-14, we select representative empathic networks from $\mathcal{W}$, taking into account the personalized needs of experts for the network structure. In addition, in order to highlight the characteristics of different types of empathic networks, we conduct some comparative analysis of the obtained networks.

\subsubsection{The most discriminating and sparse empathic networks}
According to Model 5, the most discriminating empathic network is obtained and shown in Fig.~\ref{3-1}, with the corresponding local empathic weight matrix presented below:
\begin{equation*}
W_{1}=\begin{pmatrix}
1.0000&0&0&0&0&0&0&0&0&0\\
0&0.9900&0.0100&0&0&0&0&0&0&0\\
0&0&1.0000&0&0&0&0&0&0&0\\
0&0&0&1.0000&0&0&0&0&0&0\\
0&0.9900&0&0&0.0100&0&0&0&0&0\\
0&0&0&0&0&1.0000&0&0&0&0\\
0&0&0&0&0&0&1.0000&0&0&0\\
0&0&0&0&0&0&0&1.0000&0&0\\
0&0&0&0&0&0&0&0&0.0100&0.9900\\
0&0&0&0&0&0&0&0&0&1.0000\\
\end{pmatrix}.
\end{equation*}

The sparse empathic network is obtained according to Model 6 and shown in Fig.~\ref{3-2}, with the corresponding local empathic weight matrix presented below:
\begin{equation*}
W_{2}=\begin{pmatrix}
1.0000&0&0&0&0&0&0&0&0&0\\
0&0.0100&0.9900&0&0&0&0&0&0&0\\
0&0&1.0000&0&0&0&0&0&0&0\\
0&0&0&1.0000&0&0&0&0&0&0\\
0&0&0&0&1.0000&0&0&0&0&0\\
0&0&0&0&0&1.0000&0&0&0&0\\
0&0&0&0&0&0&1.0000&0&0&0\\
0&0&0&0&0&0&0&1.0000&0&0\\
0&0&0&0&0&0&0&0&1.0000&0\\
0&0&0&0&0&0&0&0&0&1.0000\\
\end{pmatrix}.
\end{equation*}

According to Fig.~\ref{f3}, there are isolated points $d_{1}$, $d_{4}$, $d_{6}$, $d_{7}$ and $d_{8}$ in the most discriminating network, and there are links between the remaining nodes. In the sparse empathic network, In the sparse network, there is only a one-way link between nodes   $d_{2}$ and $d_{3}$.

\begin{figure}
\centering
\subfigure[The most discriminating empathic network]{
\resizebox*{4cm}{!}{\includegraphics{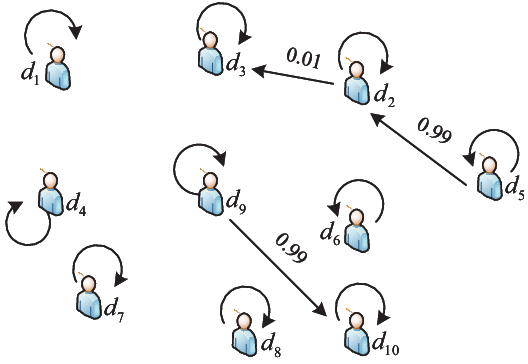}}\label{3-1}}\hspace{1cm}
\subfigure[The sparse empathic network]{
\resizebox*{4cm}{!}{\includegraphics{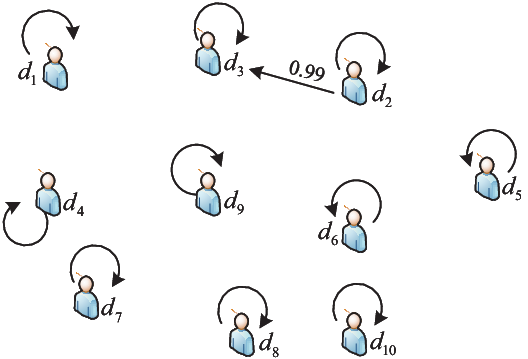}}\label{3-2}}
\caption{Comparison of the two empathic networks.}\label{f3}
\end{figure}

To further investigate the influence of different network structures on the decision results, we obtain the empathic utility matrices $U_1$ and $U_2$ of DMs over the whole set $A$ under the most discriminating and sparse empathic network structures according to equation (3).
\begin{equation*}
{U_1}=\begin{pmatrix}
0.3182&0.1818&0.1818&0.0909&0.2273\\
0.2550&0.3019&0.0711&0.2083&0.1637\\
0.2000&0.2877&0.1561&0.1123&0.2438\\
0.2696&0.1826&0.2261&0.0957&0.2261\\
0.2545&0.3023&0.0719&0.2091&0.1623\\
0.1403&0.2682&0.2256&0.1403&0.2256\\
0.2182&0.0818&0.1727&0.3091&0.2182\\
0.2459&0.2918&0.2000&0.0623&0.2000\\
0.2848&0.1909&0.1905&0.0486&0.2852\\
0.2857&0.1905&0.1905&0.0476&0.2857\\
\end{pmatrix};
{U_2}=\begin{pmatrix}
0.3182&0.1818&0.1818&0.0909&0.2273\\
0.2006&0.2878&0.1552&0.1133&0.2430\\
0.2000&0.2877&0.1561&0.1123&0.2438\\
0.2696&0.1826&0.2261&0.0957&0.2261\\
0.1444&0.3298&0.2371&0.1907&0.0980\\
0.1403&0.2682&0.2256&0.1403&0.2256\\
0.2182&0.0818&0.1727&0.3091&0.2182\\
0.2459&0.2918&0.2000&0.0623&0.2000\\
0.1918&0.2327&0.1918&0.1509&0.2327\\
0.2857&0.1905&0.1905&0.0476&0.2857\\
\end{pmatrix}.
\end{equation*}

\subsubsection{Central, distributed, and highly resilient local empathic networks}
In order to obtain the central and distributed empathic networks, the entropy based on empathic centrality is calculated as follows:
\begin{equation*}
 - \sum\limits_{j = 1}^{10} {((0.1*( \sum\limits_{j = 1}^{10}{\omega _j }))} \ln (0.1*(\sum\limits_{j = 1}^{10}{\omega _j }))).
\end{equation*}
 According to Model 7 and Model 8, the central and distributed empathic networks can be obtained by calculating the minimum and the maximum values of entropy, subject to constraints ${W^R}$. The results are shown in Fig.~\ref{4-1} and Fig.~\ref{4-2}. The empathic weight matrix corresponding to the central empathic network is presented below:
\begin{equation*}
W_{3}=\begin{pmatrix}
0.0100&0.9900&0&0&0&0&0&0&0&0\\
0&0.9900&0.0100&0&0&0&0&0&0&0\\
0&0.9900&0.0100&0&0&0&0&0&0&0\\
0&0.9900&0&0.0100&0&0&0&0&0&0\\
0&0.9900&0&0&0.0100&0&0&0&0&0\\
0&0.9900&0&0&0&0.0100&0&0&0&0\\
0&0.9900&0&0&0&0&0.0100&0&0&0\\
0&0.9900&0&0&0&0&0&0.0100&0&0\\
0&0.5934&0&0.3966&0&0&0&0&0.0100&0\\
0&0.9900&0&0&0&0&0&0&0&0.0100\\
\end{pmatrix}.
\end{equation*}

Using $W_ {3}$, we can calculate the empathic centralities of each expert, which are: $\omega _{1}=\omega _{5}=\omega _{6}=\omega _{7}=\omega _{8}=\omega _{9}=\omega _{10}=0.0100$, $\omega _{2}=9.5034$, $\omega _{3}=0.0200$ and $\omega _{4}=0.4066$. Since $\omega _{2}=9.5034>5$, according to Definition~\ref{dy1}, the obtained network is the central empathic network. $d_{2}$ is the ``opinion leader" in the network, who has more knowledge and experience and can help others to make decisions.

The weight matrix of the distributed empathic network is as follows:
\begin{equation*}
W_{4}=\begin{pmatrix}
0.0100&0&0&0.9900&0&0&0&0&0&0\\
0&0.0100&0.0100&0&0&0.1232&0.3805&0.3465&0.1298&0\\
0.3397&0&0.0100&0&0&0.4334&0&0&0.0992&0.1177\\
0.3106&0&0.1850&0.0100&0.2031&0&0&0&0.1364&0.1549\\
0&0.4950&0.1389&0&0.1670&0&0&0&0.0903&0.1088\\
0.3397&0&0&0&0&0.4434&0&0&0.0992&0.1177\\
0&0&0&0&0&0&0.6195&0&0.1810&0.1995\\
0&0&0&0&0&0&0&0.6535&0.1640&0.1825\\
0&0&0.5172&0&0.4728&0&0&0&0.0100&0\\
0&0.4950&0.1389&0&0.1570&0&0&0&0.0903&0.1188\\

\end{pmatrix}.
\end{equation*}

Using $W_{4}$,  we can also calculate the empathic  centralities of each expert, which are: $\omega _{1}=\omega _{2}=\omega _{3}=\omega _{4}=\omega _{6}=\omega _{7}=\omega _{8}=1.0000$, $\omega _{1}=1.0001$, $\omega _{5}=\omega _{10}=0.9999$, $\omega _{9}=1.0002$. Setting $\delta=0.015$, we find that $|{\omega _j} - 1| \le \delta, j=1,2,\cdots,10$, according to Definition~\ref{dy2}, this network is a distributed empathic network, and all experts in the network have equal status in the decision-making process.

\begin{figure}
\centering
\subfigure[The central empathic network]{
\resizebox*{4cm}{!}{\includegraphics{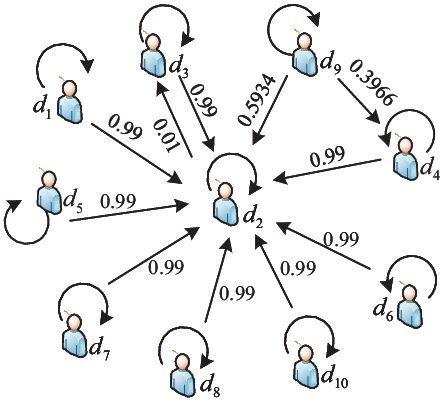}}\label{4-1}}
\subfigure[The distributed empathic network]{
\resizebox*{5cm}{!}{\includegraphics{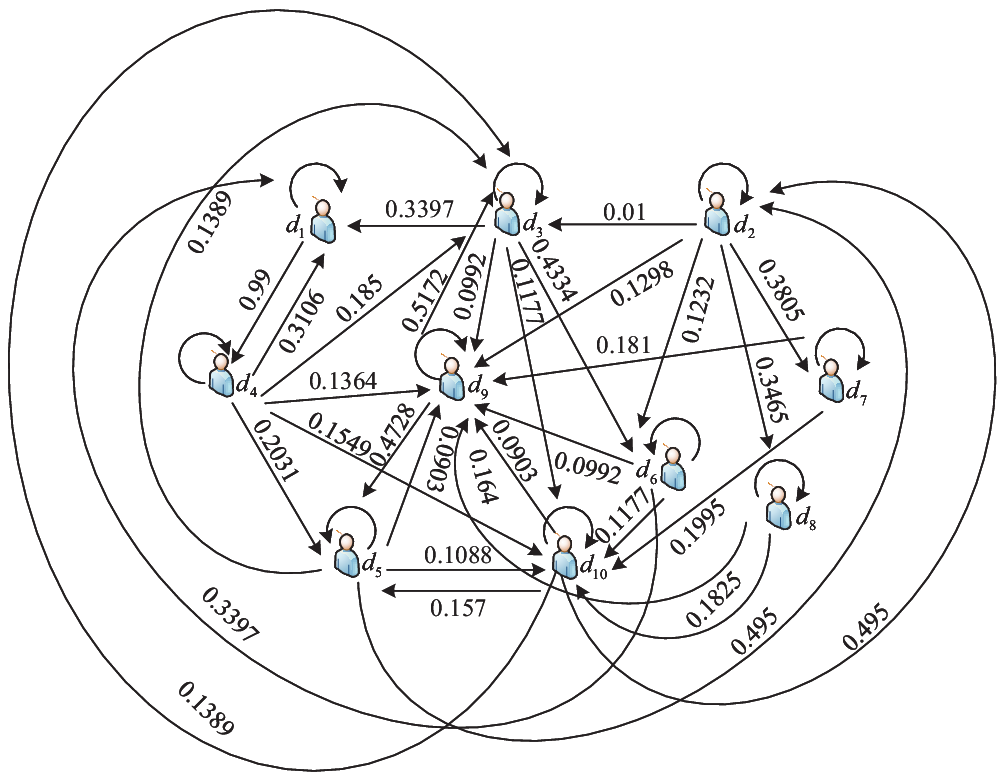}}\label{4-2}}
\subfigure[The highly resilient local empathic network]{
\resizebox*{6cm}{!}{\includegraphics{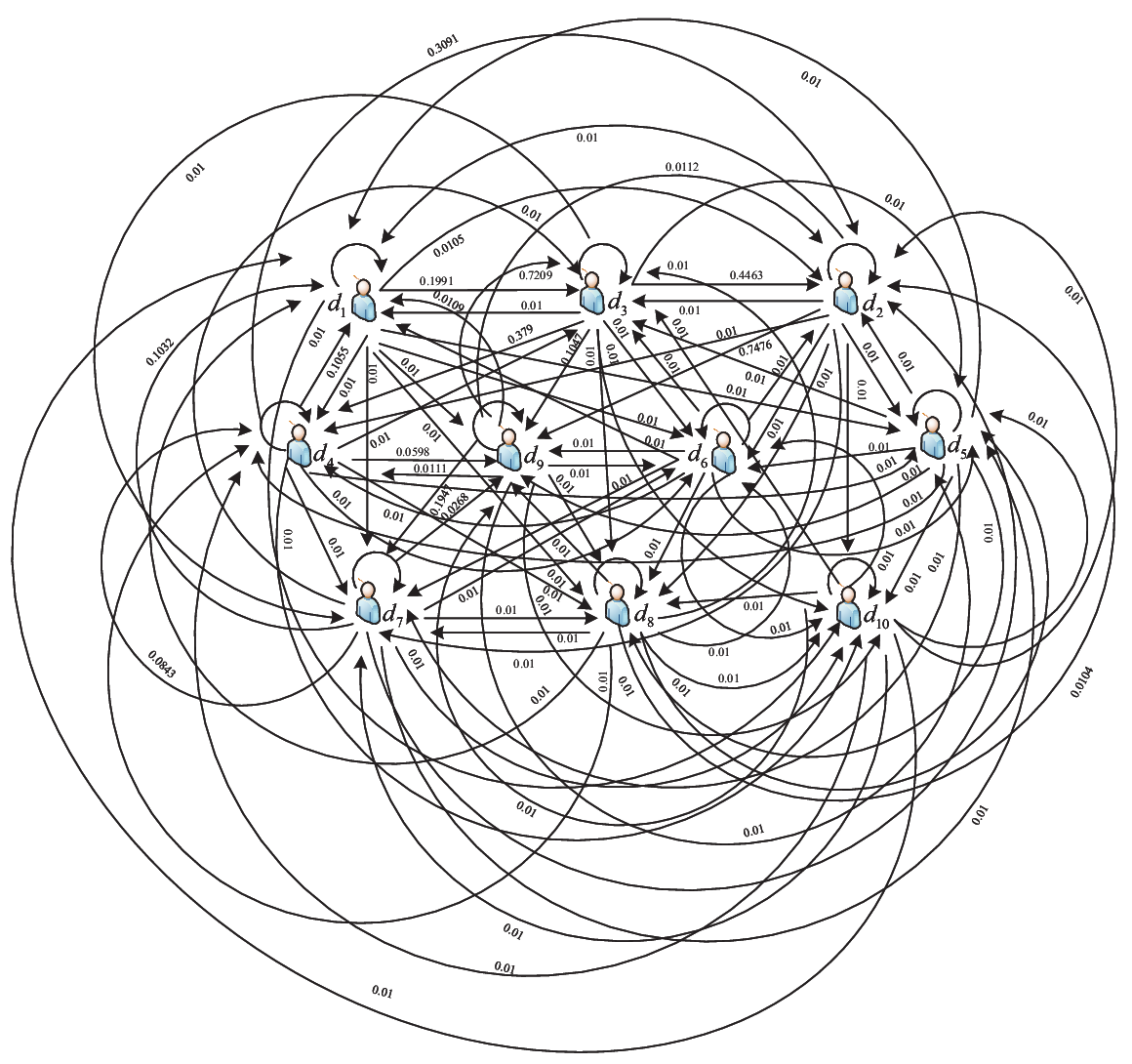}}\label{4-3}}
\caption{Central, distributed, and highly resilient local empathic networks.}\label{f4}
\end{figure}

Model 9 has been used to obtain the local empathic network with high resilience, as shown in Fig.~\ref{4-3}. The corresponding weight matrix is presented below:
\begin{equation*}
W_{5}=\begin{pmatrix}
0.7204&0.0105&0.1991&0.0100&0.0100&0.0100&0.0100&0.0100&0.0100&0.0100\\
0.0100&0.1724&0.0100&0.0100&0.0100&0.0100&0.0100&0.0100&0.7476&0.0100\\
0.0100&0.4463&0.0100&0.3790&0.0100&0.0100&0.0100&0.0100&0.1047&0.0100\\
0.1055&0.3091&0.0100&0.4655&0.0100&0.0100&0.0100&0.0100&0.0598&0.0100\\
0.0100&0.0100&0.0100&0.0100&0.9100&0.0100&0.0100&0.0100&0.0100&0.0100\\
0.0100&0.0100&0.0100&0.0100&0.0100&0.9100&0.0100&0.0100&0.0100&0.0100\\
0.1032&0.0104&0.0100&0.0843&0.0100&0.0100&0.7253&0.0100&0.0268&0.0100\\
0.0100&0.0100&0.0100&0.0100&0.0100&0.0100&0.0100&0.9100&0.0100&0.0100\\
0.0109&0.0112&0.7209&0.0111&0.0100&0.0100&0.1947&0.0100&0.0111&0.0100\\
0.0100&0.0100&0.0100&0.0100&0.0100&0.0100&0.0100&0.0100&0.0100&0.9100\\
\end{pmatrix}.
\end{equation*}

According to $W_{5}$, the empathic centralities of experts are: $\omega _{2}=\omega _{4}=0.9999$, $\omega _{1}=\omega _{3}=\omega _{5}=\omega _{6}=\omega _{7}=\omega _{8}=\omega _{9}=\omega _{10}=1.0000$. Setting the threshold for the density of the empathic network ${\rho_0}=0.9$, we find that $\rho=1\ge {\rho _0}$, $|{\omega _j} - 1| \le \delta,$ for $j=1,2,\cdots,10$. According to Definition~\ref{d}, this local empathic network is highly resilient. By comparing Figs~\ref{4-1}, ~\ref{4-2} and ~\ref{4-3}, one can conclude that there are significantly more links in the high-resilient empathic network than in the first two networks, indicating that the highly resilient empathic network has a high degree of redundancy. If one expert drops out, the remaining experts can still form a network and keep information flowing.

We use equation (3) to obtain the empathic utility matrices $U_3$, $U_4$ and $U_5$ of the experts over the whole set $A$ under the structures of the central, distributed and highly resilient local empathic network.
\begin{equation*}
{U_3}=\begin{pmatrix}
0.2562&0.3008&0.0713&0.2081&0.1635\\
0.2550&0.3019&0.0711&0.2083&0.1637\\
0.2550&0.3019&0.0711&0.2083&0.1637\\
0.2557&0.3008&0.0718&0.2082&0.1635\\
0.2545&0.3023&0.0719&0.2091&0.1623\\
0.2544&0.3017&0.0718&0.2086&0.1635\\
0.2552&0.2998&0.0712&0.2103&0.1635\\
0.2555&0.3019&0.0715&0.2078&0.1633\\
0.2605&0.2540&0.1332&0.1637&0.1887\\
0.2559&0.3009&0.0714&0.2077&0.1641\\
\end{pmatrix};
{U_4}=\begin{pmatrix}
0.2701&0.1826&0.2257&0.0957&0.2261\\
0.2150&0.2014&0.1900&0.1793&0.2144\\
0.2236&0.2264&0.2025&0.1134&0.2341\\
0.2383&0.2397&0.1914&0.1167&0.2139\\
0.2268&0.2863&0.1341&0.1699&0.1830\\
0.2230&0.2262&0.2032&0.1137&0.2340\\
0.2269&0.1308&0.1797&0.2283&0.2343\\
0.2443&0.2636&0.1969&0.0741&0.2210\\
0.1736&0.3071&0.1948&0.1498&0.1748\\
0.2282&0.2849&0.1336&0.1684&0.1848\\
\end{pmatrix};
\end{equation*}
\begin{equation*}
{U_5}=\begin{pmatrix}
0.2867&0.2072&0.1772&0.1000&0.2289\\
0.2057&0.2442&0.1714&0.1594&0.2193\\
0.2519&0.2447&0.1507&0.1550&0.1977\\
0.2619&0.2259&0.1694&0.1365&0.2062\\
0.1527&0.3203&0.2319&0.1857&0.1094\\
0.1490&0.2649&0.2216&0.1404&0.2242\\
0.2318&0.1165&0.1790&0.2534&0.2192\\
0.2440&0.2861&0.1985&0.0702&0.2012\\
0.2063&0.2441&0.1621&0.1516&0.2357\\
0.2798&0.1949&0.1900&0.0569&0.2783\\
\end{pmatrix}.
\end{equation*}

\subsubsection{Highly resilient global empathic network}
Model 10 has been used to obtain the local empathic network, as shown in Fig.~\ref{5-1}. The corresponding irreducible weight matrix is presented below:
\begin{equation*}
W_{6}=\begin{pmatrix}
0.0100&0.9900&0&0&0&0&0&0&0&0\\
0&0.0100&0.9900&0&0&0&0&0&0&0\\
0&0&0.0100&0.9900&0&0&0&0&0&0\\
0&0&0&0.0100&0.9900&0&0&0&0&0\\
0&0&0&0&0.0100&0.9900&0&0&0&0\\
0&0&0&0&0&0.0100&0.9900&0&0&0\\
0&0&0&0&0&0&0.0100&0.9900&0&0\\
0&0&0&0&0&0&0&0.0100&0.9900&0\\
0&0&0&0&0&0&0&0&0.0100&0.9900\\
0.9900&0&0&0&0&0&0&0&0&0.0100\\
\end{pmatrix}.
\end{equation*}

According to equation (7), after the transference of the empathic relationships, the global empathic weight matrix  is calculated as follows:
\begin{equation*}
G=\begin{pmatrix}
0.1046&0.1035&0.1025&0.1015&0.1005&0.0995&0.0985&0.0975&0.0965&0.0955\\
0.0955&0.1046&0.1035&0.1025&0.1015&0.1005&0.0995&0.0985&0.0975&0.0965\\
0.0965&0.0955&0.1046&0.1035&0.1025&0.1015&0.1005&0.0995&0.0985&0.0975\\
0.0975&0.0965&0.0955&0.1046&0.1035&0.1025&0.1015&0.1005&0.0995&0.0985\\
0.0985&0.0975&0.0965&0.0955&0.1046&0.1035&0.1025&0.1015&0.1005&0.0995\\
0.0995&0.0985&0.0975&0.0965&0.0955&0.1046&0.1035&0.1025&0.1015&0.1005\\
0.1005&0.0995&0.0985&0.0975&0.0965&0.0955&0.1046&0.1035&0.1025&0.1015\\
0.1015&0.1005&0.0995&0.0985&0.0975&0.0965&0.0955&0.1046&0.1035&0.1025\\
0.1025&0.1015&0.1005&0.0995&0.0985&0.0975&0.0965&0.0955&0.1046&0.1035\\
0.1035&0.1025&0.1015&0.1005&0.0995&0.0985&0.0975&0.0965&0.0955&0.1046\\
\end{pmatrix}.
\end{equation*}

Using $G$, we can calculate the empathic centralities of each expert, which are: $\omega _{1}=\omega _{2}=\cdots=\omega _{10}=1.0001$. This means that all experts have equal influence in the empathic network. In addition, the density of the empathic network is $\rho=1$. Taking again ${\rho_0}=0.9$ and $\delta=0.015$, we get $\rho \ge {\rho_0}$ and $|{\omega _j} - 1| \le \delta, \ j=1,2,\cdots,10$. According to Definition~\ref{d}, the global empathic network obtained from Model 10 is an empathic network with high resilience. Fig.~\ref{5-2} shows the global empathic network. It can be observed that when the empathic relationships are not transmitted, the local empathic network is relatively sparse. After the transference of empathic relationships, the corresponding global empathic network is very dense and the organizational structure is flat, which further indicates that the global empathic network is highly resilient.

\begin{figure}
\centering
\subfigure[The local empathic network]{
\resizebox*{4cm}{!}{\includegraphics{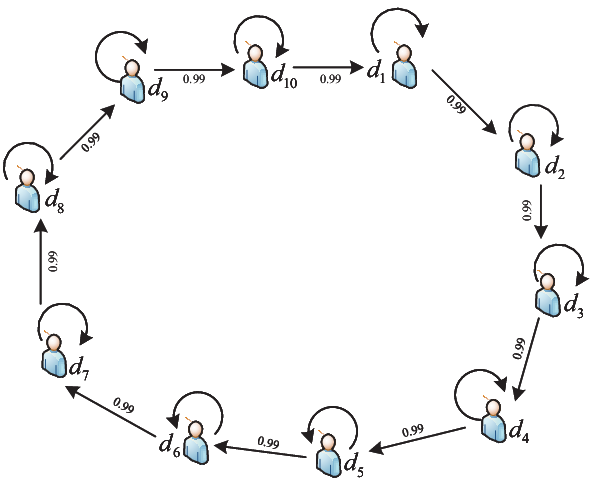}}\label{5-1}}\hspace{1cm}
\subfigure[The highly resilient global empathic network]{
\resizebox*{6cm}{!}{\includegraphics{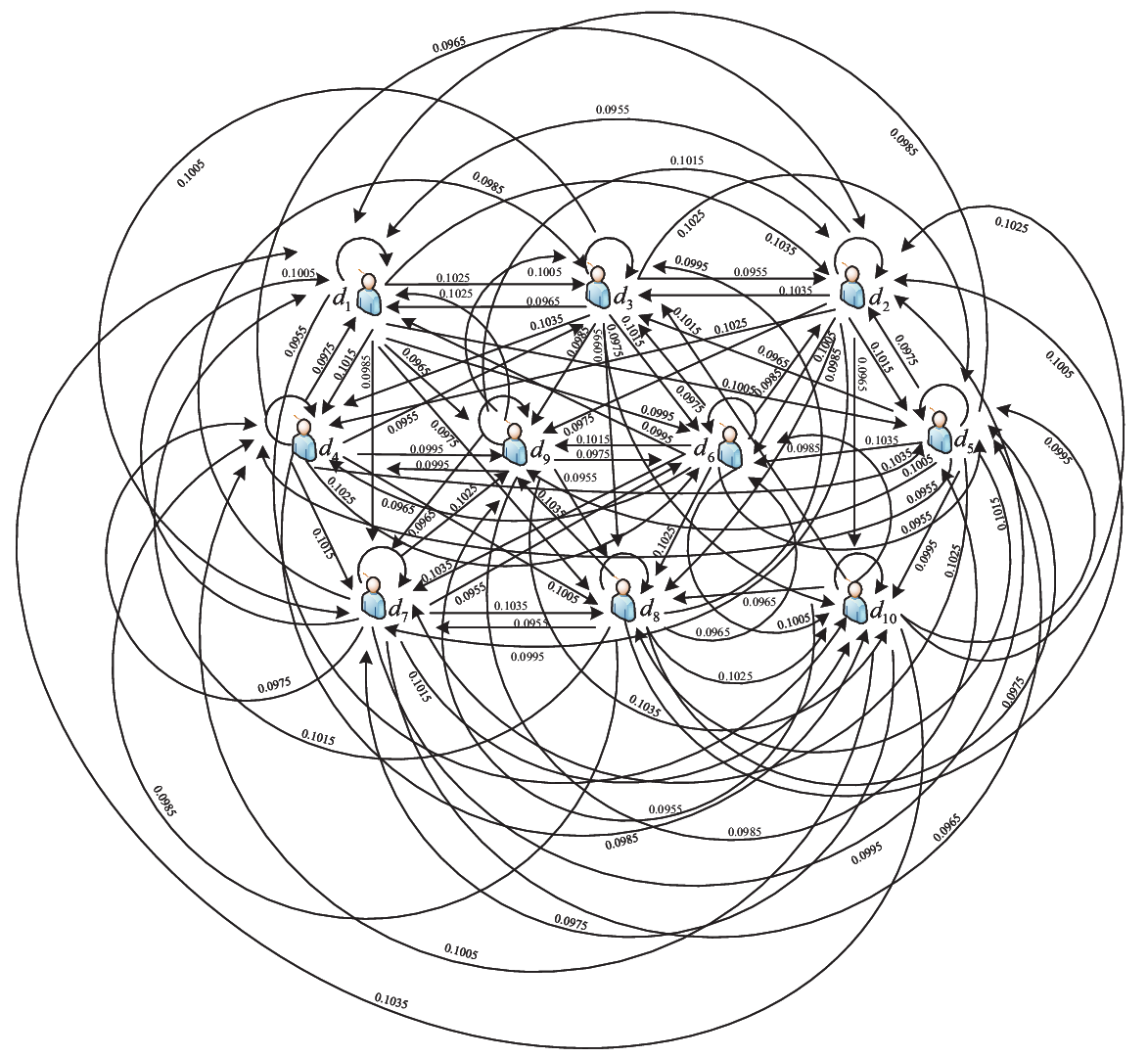}}\label{5-2}}
\caption{Local and corresponding highly resilient global empathic networks.}\label{f5}
\end{figure}

We use equation (6) to obtain the global empathic utility matrix $U_6$ of the experts over the whole set $A$ under the highly resilient global empathic network structure.
\begin{equation*}
{U_6}=\begin{pmatrix}
0.2273&0.2353&0.1847&0.1411&0.2116\\
0.2264&0.2358&0.1847&0.1417&0.2114\\
0.2261&0.2352&0.1859&0.1410&0.2119\\
0.2264&0.2346&0.1862&0.1413&0.2116\\
0.2260&0.2352&0.1858&0.1417&0.2115\\
0.2268&0.2342&0.1853&0.1412&0.2126\\
0.2277&0.2338&0.1849&0.1412&0.2125\\
0.2278&0.2354&0.1850&0.1395&0.2124\\
0.2276&0.2348&0.1848&0.1403&0.2125\\
0.2279&0.2348&0.1848&0.1402&0.2123\\
\end{pmatrix}.
\end{equation*}

\subsubsection{Star, bus, and tree empathic networks}
Empathic networks with star, bus and tree topologies are calculated according to Models 12-14, as shown in Fig.~\ref{f6}. Observe that the bus-type local empathic network presented in Fig.~\ref{f6}b is equivalent to the network from Fig.~\ref{5-1} after cutting the link between nodes 10 and 1. In order to make comparison of the bus-type local empathic network with the global empathic network with high resilience, the local empathic weight matrix $W_{7}$ has been used for the calculation of the global empathic weight matrix $G^\prime$ according to  equation (7),
\begin{equation*}
W_{7}=\begin{pmatrix}
0.9900&0.0100&0&0&0&0&0&0&0&0\\
0&0.9900&0.0100&0&0&0&0&0&0&0\\
0&0&0.9900&0.0100&0&0&0&0&0&0\\
0&0&0&0.9900&0.0100&0&0&0&0&0\\
0&0&0&0&0.9900&0.0100&0&0&0&0\\
0&0&0&0&0&0.9900&0.0100&0&0&0\\
0&0&0&0&0&0&0.9900&0.0100&0&0\\
0&0&0&0&0&0&0&0.9900&0.0100&0\\
0&0&0&0&0&0&0&0&0.0100&0.9900\\
0&0&0&0&0&0&0&0&0&1.0000\\
\end{pmatrix},
\end{equation*}

\begin{equation*}
G^\prime=\begin{pmatrix}
0.9900&0.0099&0.0001&0&0&0&0&0&0&0\\
0&0.9900&0.0099&0.0001&0&0&0&0&0&0\\
0&0&0.9900&0.0099&0.0001&0&0&0&0&0\\
0&0&0&0.9900&0.0099&0.0001&0&0&0&0\\
0&0&0&0&0.9900&0.0099&0.0001&0&0&0\\
0&0&0&0&0&0.9900&0.0099&0.0001&0&0\\
0&0&0&0&0&0&0.9900&0.0099&0&0.0001\\
0&0&0&0&0&0&0&0.9900&0.0001&0.0099\\
0&0&0&0&0&0&0&0&0.0100&0.9900\\
0&0&0&0&0&0&0&0&0&1.0000\\
\end{pmatrix}.
\end{equation*}

\begin{figure}[!h]
\centering
\subfigure[The star topology]{
\resizebox*{4cm}{!}{\includegraphics{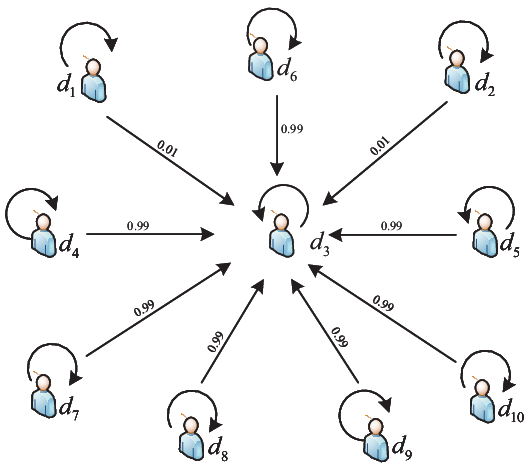}}\label{6-1}}\hspace{1cm}
\subfigure[The bus topology]{
\resizebox*{4.5cm}{!}{\includegraphics{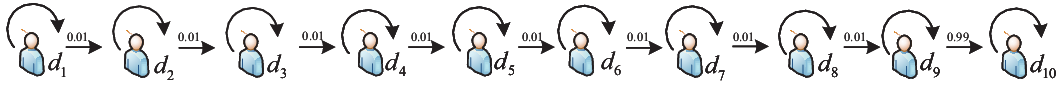}}\label{6-2}}\hspace{1cm}
\subfigure[The tree topology]{
\resizebox*{3cm}{!}{\includegraphics{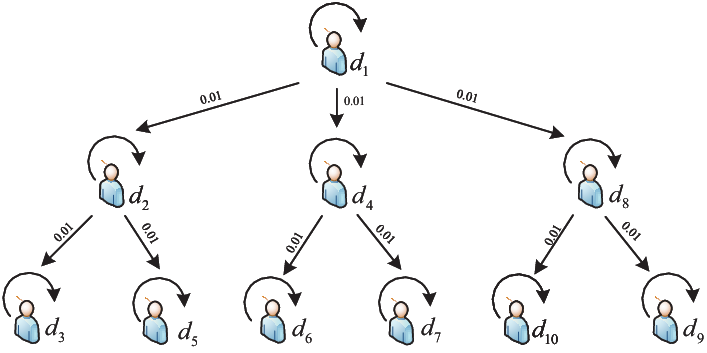}}\label{6-3}}
\caption{Different topologies of local empathic network.}\label{f6}
\end{figure}

\begin{figure}[!h]
\centering
\includegraphics[width=0.3\textwidth]{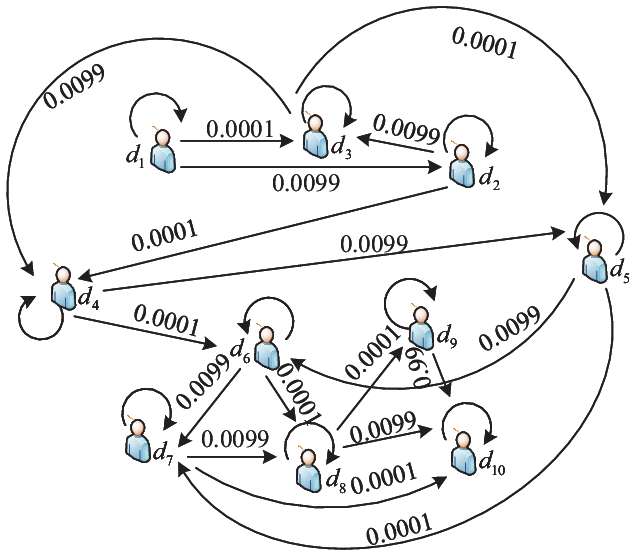}
\caption{The global empathic network corresponding to the bus-type local empathic network.}\label{f7}
\end{figure}

Fig.~\ref{f7} shows the global empathic network corresponding to the bus-type local empathic network. One can observe that global empathic network corresponding to the bus-type local empathic network from Fig.~\ref{5-2} has fewer network links and is significantly less resilient than the global empathic network with high resilience. This further illustrates the importance of loops in local empathic networks for high-resilient global empathic networks.

We use equation (3) to obtain the empathic utility matrices $U_7$, $U_8$ and $U_9$ of the experts over the whole set $A$ under the structures of the star, bus, and tree local empathic network.
\begin{equation*}
{U_7}=\begin{pmatrix}
0.3170&0.1829&0.1815&0.0911&0.2275\\
0.2550&0.3019&0.0711&0.2083&0.1637\\
0.2000&0.2877&0.1561&0.1123&0.2438\\
0.2007&0.2866&0.1568&0.1121&0.2436\\
0.1994&0.2881&0.1569&0.1131&0.2423\\
0.1994&0.2875&0.1568&0.1126&0.2436\\
0.2002&0.2856&0.1563&0.1143&0.2435\\
0.2005&0.2877&0.1565&0.1118&0.2434\\
0.1999&0.2872&0.1565&0.1127&0.2437\\
0.2009&0.2867&0.1564&0.1117&0.2442\\
\end{pmatrix};
{U_8}=\begin{pmatrix}
0.3176&0.1830&0.1807&0.0921&0.2267\\
0.2550&0.3019&0.0711&0.2083&0.1637\\
0.2007&0.2866&0.1568&0.1121&0.2436\\
0.2683&0.1841&0.2262&0.0966&0.2248\\
0.1444&0.3292&0.2370&0.1902&0.0993\\
0.1411&0.2663&0.2251&0.1420&0.2255\\
0.2185&0.0839&0.1730&0.3066&0.2180\\
0.2454&0.2912&0.1999&0.0632&0.2003\\
0.2848&0.1909&0.1905&0.0486&0.2852\\
0.2857&0.1905&0.1905&0.0476&0.2857\\
\end{pmatrix};
\end{equation*}
\begin{equation*}
{U_9}=\begin{pmatrix}
0.3164&0.1841&0.1813&0.0918&0.2264\\
0.2539&0.3021&0.0727&0.2081&0.1631\\
0.2000&0.2877&0.1561&0.1123&0.2438\\
0.2678&0.1824&0.2256&0.0983&0.2260\\
0.1444&0.3298&0.2371&0.1907&0.0980\\
0.1403&0.2682&0.2256&0.1403&0.2256\\
0.2182&0.0818&0.1727&0.3091&0.2182\\
0.2458&0.2902&0.1998&0.0630&0.2012\\
0.1918&0.2327&0.1918&0.1509&0.2327\\
0.2857&0.1905&0.1905&0.0476&0.2857\\
\end{pmatrix}.
\end{equation*}

\subsection{Comparison of decision results under different network structures}\label{decision}
As mentioned above, network structure is an important factor affecting decision-making. Therefore, we study the impact of different network structures on the outcome of emergency decision-making in this subsection. Salehi-Abari et al. \cite{10A.Salehi-Abari2019} proposed the equation of social welfare $sw(\bm{u},a_s)$ of alternative $a_s \in A$ as follows:
\begin{equation}
sw(\bm{u},a_s) = \sum\limits_{j \in N} {{u_j}(a_s)},
\end{equation}
where $\bm{u}=({u_1}( \cdot ),{u_2}( \cdot ), \cdots ,{u_n}( \cdot ))$ is a vector specifying the utility function of each expert.

Experts can choose the alternative with the maximum social welfare as the best alternative $a^*$. That is
\begin{equation}
{a^*} = \arg \mathop {\max }\limits_{{a_s} \in A} sw({\bf{u}},{a_s}).
\end{equation}

According to the matrices ${U_1}$ to ${U_9}$ obtained earlier, equation (13) is used to calculate the social welfare of all alternatives under different network structures. The best alternative $a^*$ can be obtained by equation (14). We show the social welfare under different network structures as well as the best alternative $a^*$ in Table~\ref{table1}. ``Without considering network" in Table ~\ref{table1} refers to the calculation of social welfare directly based on the intrinsic utilities of experts without considering the empathic relationships between experts. According to Table~\ref{table1}, we can find that the best alternative may be different under different network structures, and the best alternative exists among $a_1$, $a_2$ and $a_5$. These results confirm that the method proposed in this paper can infer the specific network structure under different emergency scenarios, thus assisting experts to obtain more objective and scientifically based decision aiding.

\renewcommand\arraystretch{0.95}
\begin{table}[!h]
\scriptsize
\caption{Social welfare $sw$ and best alternative $a^*$ under different network structures}
\label{table1}
\setlength\tabcolsep{15pt}
\centering
\begin{tabular*}{15cm}{ccccccc}
\hline
Network structures&$a_1$&$a_2$&$a_3$&$a_4$&$a_5$&$a^*$\\
\hline
Without considering network
&2.2697&\textbf{2.3489}&1.8519&1.4091&2.1203&$a_2$\\
The most discriminating empathic network
&\textbf{2.4722}&2.2795&1.6862&1.3243&2.2378&$a_1$\\
Sparse empathic network
&2.2147&\textbf{2.3347}&1.9369&1.3131&2.2004&$a_2$\\
Central empathic network
&2.5581&\textbf{2.9658}&0.7762&2.0401&1.6598&$a_2$\\
Distributed empathic network
&2.2697&\textbf{2.3489}&1.8519&1.4091&2.1203&$a_2$\\
Highly resilient local empathic network
&2.2696&\textbf{2.3489}&1.8519&1.4091&2.1203&$a_2$\\
Highly resilient global empathic network
&2.2699&\textbf{2.3491}&1.8521&1.4092&2.1205&$a_2$\\
Star empathic network
&2.1730&\textbf{2.7820}&1.5049&1.1999&2.3394&$a_2$\\
Bus empathic network
&\textbf{2.3614}&2.3076&1.8507&1.3074&2.1728&$a_1$\\
Tree empathic network
&2.2642&\textbf{2.3496}&1.8532&1.4122&2.1206&$a_2$\\
\hline
\end{tabular*}
\end{table}

The comparative analysis between the studies of this paper and relevant literatures is shown in Table~\ref{table2}. Salehi-Abari et al. \cite{10A.Salehi-Abari2019} proposes empathic networks for the first time and presents local and global empathic social choice models. Xu et al.\cite{4X.Xu2020} proposes a new consensus reaching framework based on the social network in large-group emergency decision-making. It is directly assumed in this method that the network structure is known. However, we take into account the scarcity of data in emergency decision making and the network structure may be incomplete. Inspired by the preference disaggregation, an empathetic network learning method is proposed to infer the complete empathetic network structure based on the indirect preference information and node information provided by DMs to assist the emergency decision. Generally, literatures often use subjective scoring method to obtain the complete network structure, that is, the social relationships between decision makers are given by the mutual evaluation matrix created by themselves, such as literature \cite{3S.M.Yu2021}. The drawback of this method is that requiring DMs to provide the evaluation matrix directly may be burdensome to them and is too subjective. In addition, in view of the incomplete trust network structure, Xu et al.\cite{42Y.Xu2021} and Sun et al.\cite{2J.Sun2023} constructed uncertainty optimization models and dynamic trust propagation mode based on the existing network node information to complete the trust network. The defect of these methods is that they ignore that the DM's preferences about the alternatives will be affected by the neighbors with whom they have a trust relationship. Therefore, the method proposed in this paper takes into account the indirect preference information provided by DMs after they are influenced by neighbors.

\renewcommand\arraystretch{1}
\begin{table}[!h]
\scriptsize
\caption{The comparison of the related literatures.}
\label{table2}
\setlength\tabcolsep{9pt}
\centering
\begin{tabular*}{16cm}{p{2cm}<{\centering}p{2cm}<{\centering}p{2cm}<{\centering}p{2cm}<{\centering}p{2cm}<{\centering}
p{2.5cm}<{\centering}}
\hline
References&Types of social networks&Is the network structure complete?&Indirect preference information&Node information&Methods\\
\hline
\cite{10A.Salehi-Abari2019}&Empathic networks&$\checkmark$&$\times$&$\times$&$\times$\\
\cite{4X.Xu2020}&Trust networks&$\checkmark$&$\times$&$\times$&$\times$\\
\cite{3S.M.Yu2021}&Trust networks&$\times$&$\times$&$\times$&A mutual evaluation matrix\\
\cite{2J.Sun2023}&Trust networks&$\times$&$\times$&$\checkmark$&A dynamic trust propagation mode\\
\cite{42Y.Xu2021}&Trust networks&$\times$&$\times$&$\checkmark$&Uncertainty optimization models\\
This paper&Empathic networks&$\times$&$\checkmark$&$\checkmark$&The empathic network learning method\\
\hline
\end{tabular*}
\end{table}

\section{ Conclusion}
Emergency decision-making is today a hot topic of operational research and management science. Given the pressure of time constraints, the scarcity of information and resources, and the constrained cognitive capacities of experts, the decision information provided is often inconsistent and incomplete. Using the robust ordinal regression approach based on preference disaggregation, we introduced in this paper a novel empathic network learning methodology designed to address the challenge of inferring the empathic network influencing the decision-making behavior of experts in the context of incomplete decision information. The primary findings and innovations of this paper are outlined as follows:

(1)	The method proposed in this paper only requires the DMs to provide a holistic partial preference information, such as pairwise comparisons of some reference alternatives or pairwise comparisons of the intensities of the empathic relationships between network nodes representing experts. First, we used robust ordinal regression approach to construct a model based on independent holistic preference information to complete the incomplete judgment matrix of each expert and calculate their intrinsic utilities. Then, the partial empathic preference information and the node information on the reference set were translated into mathematical constrains defining the compatible empathic networks able to reproduce the decision information. Moreover, in the case of inconsistent decision information, a mixed $\{0,1\}$ linear programming model was built to find and eliminate a minimum number of inconsistent constraints.

(2)	Based on the set of compatible empathic networks, we defined the necessary and possible empathic relationships between DMs, and formulated models to calculate the binary empathic relationships between any two DMs. Results of these models presented to DMs can deepen the DMs' understanding of the decision problem and stimulate them for providing new preference and node information. This interactive decision-making process promotes rational and intelligible decision-making.

(3) Since in some emergency environments the DMs may need to see more clear outcomes than just necessary and possible, we designed six types of target networks for different needs and we built respective models to assist the DMs in selecting the most representative empathic network from the set of compatible empathic networks. All compatible empathic networks played a role in the selection process ensuring robustness of the decision.

(4)	The proposed method was validated on a hypothetical case study concerning selection of the best material for emergency shelters. The results of this case study show that this method can not only infer the empathic network able to reproduce the decision information representing the prior knowledge of the DMs, but also assist them in obtaining the most representative empathic network. The case study illustrates a rich spectrum of possible insights into the decision problem and demonstrates the computational efficiency of this method.

In the era of big data, people have increasingly higher requirements for data integrity. The method proposed in this paper can learn the empathic network and restore the integrity of information in the case of incomplete and inconsistent decision information. It provides a new framework for network structure inference in the field of emergency decision-making. The empathic network compatible with incomplete preference and node information coming from DMs is, moreover, an excellent basis for advanced research on consensus reaching, risk analysis and strategy selection in emergency decision-making considered recently, for example, in \cite{43R.Ding2020, 44J.Cao2022, 45X.Yin2021}. In the future research, we intend to generalize our methodology to other types of social networks, such as trust networks.

\section*{CRediT authorship contribution statement}
\textbf{Simin Shen:} Writing - original draft, Conceptualization, Methodology, Formal analysis. \textbf{Zaiwu Gong:} Conceptualization, Methodology, Writing - review \& editing, Funding acquisition. \textbf{Bin Zhou:} Conceptualization, Methodology, Validation, Writing - review \& editing. \textbf{Roman S\l owi\'nski:} Conceptualization, Methodology, Writing - review \& editing, Funding acquisition.

\section*{Declaration of competing interest}
The authors declare that they have no known competing financial interests or personal relationships that could have appeared to
influence the work reported in this paper.

\section*{Data availability}
Data will be made available on request.

\section*{Acknowledgments} This research was partially supported by the National Natural Science Foundation of China (grant number 71971121), the SBAD funding from the Polish Ministry of Education and Science, the Major Project Plan of Philosophy and Social Sciences Research at Jiangsu University, China (grant number 2020SJZDA076), and the Startup Foundation for Introducing Talent of NUIST.

\section*{Appendix A}
The proof of Theorem~\ref{T1}.
\begin{proof}
The model of entropy maximization is as follows:

\begin{align}
\begin{array}{r@{~}l}
\quad &\max \ -\sum\limits_{j = 1}^n {\omega _j^\prime} \ln \omega _j^\prime \\[10pt]
\quad &\text{\textup{subject to }} \sum\limits_{j = 1}^n {\omega _j^\prime}  = 1\\[10pt]\tag{A.1}
\end{array}
\end{align}

Construct the Lagrange function $L(\omega _1^\prime,\omega _2^\prime, \cdots ,\omega _n^\prime,\lambda ) = \sum\limits_{j = 1}^n {\omega _j^\prime} \ln \omega _j^\prime + \lambda (\sum\limits_{j = 1}^n {\omega _j^\prime}  - 1)$.
Take the partial derivatives with respect to ${\omega _j^\prime}$ and $\lambda$:

\begin{align}
\begin{array}{r@{~}l}
\quad &\frac{{\partial L}}{{\partial \omega _j^\prime}} = \ln \omega _j^\prime + \lambda  + 1,\tag{A.2}
\end{array}
\end{align}
\begin{align}
\begin{array}{r@{~}l}
\quad &\frac{{\partial L}}{{\partial \lambda }} = \sum\limits_{j = 1}^n {\omega _j^\prime}  - 1.\tag{A.3}
\end{array}
\end{align}

Let $\frac{{\partial L}}{{\partial \omega _j^\prime}} = \frac{{\partial L}}{{\partial \lambda }} = 0$. When $\omega _j^\prime = \frac{1}{n}$, that is, ${\omega _j} = 1, \ j = 1,2, \cdots ,n$,  the entropy reaches maximum value $\ln n$.

Let $g(\omega _j^\prime) =  - \omega _j^\prime\ln \omega _j^\prime$, then ${g^{\prime \prime }}\left( {\omega _j^\prime} \right) = \frac{{d( - \ln \omega _j^\prime - 1)}}{{d(\omega _j^\prime)}} =  - \frac{1}{{\omega _j^\prime}} < 0$.

When $\omega _j^\prime \to 0$, $g'(\omega _j^\prime) =  - \ln \omega _j^\prime - 1 > 0$, then $g(\omega _j^\prime)$ is a concave function on the domain $\omega _j^\prime \in [0,1]$, and the minimum value must be taken on the boundaries. That is $g(0) = 0,g(1) = 0$.  Therefore, when $\omega _k^\prime = 1, \ \omega _j^\prime = 0, \ j \ne k$, the entropy reaches the minimum value 0.
\end{proof}

\section*{Appendix B}
\begin{lemma}\label{l1}\cite{10A.Salehi-Abari2019}
Assuming $W$ be the local empathic weight  matrix and $D$  the diagonal matrix formed by the diagonal elements of $W$, let $I - W + D = I - B$. Then $\rho (B) < 1$.
\end{lemma}

\begin{lemma}\label{l2}\cite{41A.Berman1994}
Assuming $A \in {C^{n \times n}}$ is irreducible, then $B = A - diag({a_{11}},{a_{22}}, \cdots ,{a_{nn}})$ is also irreducible.
\end{lemma}

\begin{lemma}\label{l3}\cite{46C.D.Meyer2023}
Given $A \in {R^{n \times n}}$ is a nonnegative matrix. If $\rho (A) < 1$, then $A$ is irreducible if and only if ${(I - A)^{ - 1}} > 0$.
\end{lemma}

\begin{lemma}\label{l4}
Given $A,B\ge 0 \in {C^{n \times n}}$, and $A$ is an irreducible matrix, then $A + B$ is also irreducible.
\end{lemma}
\begin{proof}
Assuming $A+B$ is reducible, using Definition~\ref{d1}, there is a permutation matrix $P$, such that
\begin{align}
\begin{array}{r@{~}l}
PA{P^T} + PB{P^T} = \left( {\begin{array}{*{20}{c}}
E&C\\
0&D
\end{array}} \right), \ E \in {C^{r \times r}}, \ D \in {C^{(n - r) \times (n - r)}}.\tag{B.1}
\end{array}
\end{align}

Since both $A$ and $B$ are nonnegative matrices, both $PA{P^T}$ and $PB{P^T}$ must have the form at the right end of the above equation.
This contradicts the fact that $A$ is an irreducible matrix!
Therefore, $A + B$ is irreducible.
\end{proof}

The proof of Theorem~\ref{T2}.
\begin{proof}
Let's prove the necessity first.
Let $I-W+D=I-B$. Using Lemma~\ref{l1}, we have $\rho (B) < 1$.
Since $B = W-D$ and $W$ is irreducible, using Lemma~\ref{l2}, $B$ is also irreducible.
As $B \ge 0$ and $\rho (B) < 1$, using Lemma~\ref{l3},
we know that ${(I - B)^{ - 1}} = {(I - W + D)^{ - 1}} > 0$. Thus, every element of ${(I - W + D)^{ - 1}}$ is greater than 0, while $D$ is a diagonal matrix whose diagonal elements are ${w_{jj}} > 0,\ \forall j \in N$.
In particular, $G ={(I - W + D)^{-1}} D$ is each column entry of ${(I - W + D)^{ - 1}}$ multiplied by ${w_{jj}}, \ \forall j \in N$.
As a consequence, each element of $G$ is greater than 0, that is, $G = {(I-W + D)^{-1}}D>0$.

Then, let's prove sufficiency.
If $G = {(I-W+D)^{-1}}D>0$ and the diagonal entries of $D$ are all greater than 0,
then it follows ${(I - W + D)^{ - 1}}>0$.
Let $I-W+D=I-B$, then using Lemma~\ref{l1} we have $\rho (B) < 1$.
From Lemma~\ref{l3}, $B$ is irreducible.
While $B=W-D$, then $W=B+D$. By Lemma~\ref{l4}, it follows that $W$ is also irreducible.
\end{proof}

\begin{lemma}\label{l5}\cite{46C.D.Meyer2023}
$A$ is irreducible if and only if the direct graph $G$ is a strongly connected graph.
\end{lemma}

\begin{lemma}\label{l6}\cite{47W.T.Tutte2001}
A directed graph is strongly connected if and only if it has a loop that passes through each node at least once.
\end{lemma}

\begin{lemma}\label{l7}
$A$ is irreducible if and only if the direct graph $G$ has a loop that passes through each node at least once.
\end{lemma}

\begin{proof}
It is proved with the help of Lemma~\ref{l5} and Lemma~\ref{l6}.
\end{proof}

The proof of Corollary~\ref{C1}.
\begin{proof}
The empathic network corresponding to the local empathic weight matrix $W$ has only more self-loops at each node than its direct graph, but when the empathic network has a loop its direct graph also has this loop. Then, the proof is obtained with the help of Lemma~\ref{l7} and Theorem~\ref{T2}.
\end{proof}

The proof of Corollary~\ref{C2}.
\begin{proof}
According to the sufficient condition of Corollary~\ref{C1}, we know that the empathic network corresponding to the empathic weight matrix $W$ has a loop. Thus, we obtain the form of $W$.
\end{proof}

\end{document}